\documentclass[a4paper]{amsart}

\usepackage{cite}
\usepackage{amsmath}
\usepackage{amssymb}
\usepackage{amsthm}   
\usepackage{empheq}   
\usepackage{graphicx}
\usepackage{color}
\usepackage{appendix}
\usepackage{hyperref}

\newtheorem{theorem}{Theorem}[section]
\newtheorem{lemma}[theorem]{Lemma}
\newtheorem{proposition}[theorem]{Proposition}
\newtheorem{corollary}[theorem]{Corollary}

\theoremstyle{definition}
\newtheorem{definition}[theorem]{Definition}
\newtheorem{example}[theorem]{Example}

\theoremstyle{remark}
\newtheorem{remark}[theorem]{Remark}

\numberwithin{equation}{section}

\def\EE{{\mathcal{E}}}
\def\FF{{\mathcal{F}}}
\def\tt{{\mathtt{t}}}
\def\ss{{\mathtt{s}}}
\def\SS{{\mathbf{S}}}
\def\TT{{\mathbf{T}}}

\begin{document}

\title[Regular Dirichlet extension]{Regular Dirichlet extensions of one-dimensional Brownian motion}

\author{ Liping Li}
\address{Institute of Applied Mathematics, Academy of Mathematics and Systems Science, Chinese Academy of Sciences, Beijing 100190, China.}
\email{liping\_li@amss.ac.cn}
\thanks{The first named author is partially supported by a joint grant (No. 2015LH0043) of China Postdoctoral Science Foundation and Chinese Academy of Science. The second named author is partially supported by NSFC No. 11271240.}

\author{Jiangang Ying}
\address{School of Mathematical Sciences, Fudan University, Shanghai 200433, China.}
\email{jgying@fudan.edu.cn}


\date{\today}



\begin{abstract}
The regular Dirichlet extension is the dual concept of regular Dirichlet subspace. The main purpose of this paper is to characterize all the regular Dirichlet extensions of one-dimensional Brownian motion and to explore their structures. It is shown that every regular Dirichlet extension of one-dimensional Brownian motion may essentially decomposed into at most countable disjoint invariant intervals and an $\EE$-polar set relative to this regular Dirichlet extension. On each invariant interval the regular Dirichlet extension is characterized uniquely by a scale function in a given class. To explore the structure of regular Dirichlet extension we apply the idea introduced in \cite{LY14}, we formulate the trace Dirichlet forms and attain the darning process associated with the restriction to each invariant interval of the orthogonal complement of $H^1_\mathrm{e}(\mathbb{R})$ in the extended Dirichlet space of the regular Dirichlet extension. As a result, we find an answer to a long-standing problem whether a pure jump Dirichlet form has proper regular Dirichlet subspaces.
\end{abstract}

\maketitle

\tableofcontents

\section{Introduction}\label{SEC1}

The notion of regular Dirichlet subspace (or simply regular subspace) was first raised by the second author and his co-authors in \cite{FFY05}. Roughly speaking, it is a subspace of a Dirichlet space but also a regular Dirichlet form on the same state space. Precisely, let $E$ be a locally compact separable metric space and $m$ a fully supported measure on $E$. If two regular Dirichlet forms $(\EE^1,\FF^1)$ and $(\EE^2,\FF^2)$ on $L^2(E,m)$ satisfy
\[
	\FF^1\subset \FF^2, \quad \EE^2(u,v)=\EE^1(u,v),\quad \forall u,v\in \FF^1,
\]
then $(\EE^1,\FF^1)$ is called a regular Dirichlet subspace of $(\EE^2,\FF^2)$. It is called a proper one provided $\FF^1\neq \FF^2$. A complete characterization for regular Dirichlet subspaces of one-dimensional Brownian motion was given in \cite{FFY05}. To make it clear, consider $(\EE^2,\FF^2)=\left(\frac{1}{2}\mathbf{D}, H^1(\mathbb{R})\right)$, where $H^1(\mathbb{R})$ is the $1$-Sobolev space and $\mathbf{D}$ is the Dirichlet integral, i.e., for any $u,v\in H^1(\mathbb{R})$,
\[
	\mathbf{D}(u,v)=\int_{\mathbb{R}}u'(x)v'(x)dx.
\]
 It is well-known that the associated Markov process of $\left(\frac{1}{2}\mathbf{D}, H^1(\mathbb{R})\right)$ is indeed the one-dimensional Brownian motion, which is denoted by $B=(B_t)_{t\geq 0}$ hereafter. Then any regular Dirichlet subspace  $(\EE^1, \FF^1)$ of $(\frac{1}{2}\mathbf{D}, H^1(\mathbb{R}))$ corresponds to an irreducible diffusion process on $\mathbb{R}$ with no killing inside, the speed measure $m$ (Lebesgue measure) and the scale function $\ss$ in the following class:
\begin{equation}\label{EQ1SRS}
\mathbf{S}(\mathbb{R})=\{\ss: \mathbb{R}\rightarrow \mathbb{R}, \text{ strictly increasing and absolutely continuous}, \ss'=0\text{ or }1\}.
\end{equation}
Furthermore, $(\EE^1,\FF^1)$ may be written as
\[
\begin{aligned}
	&\FF^1=\left\{u\in L^2(\mathbb{R}): u\ll \ss, du/d\ss\in L^2(\mathbb{R},d\ss)\right\}, \\
	&\EE^1(u,v)=\frac{1}{2}\int_\mathbb{R}\frac{du}{d\ss}\frac{dv}{d\ss}d\ss,\quad u,v\in \FF^1,
\end{aligned}
\]
where the notation $u\ll \ss$ means that $u$ is absolutely continuous with respect to $\ss$.

In this paper, we shall consider the dual notion of regular Dirichlet subspace. Its formal definition is as follows.

\begin{definition}
Let $E$ be a locally compact separable metric space and $m$ a fully supported Radon measure on $E$. Given two regular Dirichlet forms $(\EE^1,\FF^1)$ and $(\EE^2,\FF^2)$ on $L^2(E,m)$, $(\EE^2,\FF^2)$ is said to be a \emph{regular Dirichlet extension} (or simply regular extension) of $(\EE^1,\FF^1)$ if
\begin{equation}\label{EQ1FFE}
	\FF^1\subset \FF^2, \quad \EE^2(u,v)=\EE^1(u,v),\quad \forall u,v\in \FF^1.
\end{equation}
\end{definition}

In other words, $(\EE^2,\FF^2)$ is a regular Dirichlet extension of $(\EE^1,\FF^1)$ if and only if $(\EE^1,\FF^1)$ is a regular Dirichlet subspace of $(\EE^2,\FF^2)$.
Naturally, given a fixed regular Dirichlet form, the basic problems for this new notion are
\begin{itemize}
\item[\textbf{(Q.1)}] whether the proper regular Dirichlet extensions exist;
\item[\textbf{(Q.2)}] if exist, how to characterize all of them;
\item[\textbf{(Q.3)}] how to describe their structures.
\end{itemize}

We shall focus on regular extensions of one-dimensional Brownian motion in this paper, 
more precisely regular Dirichlet extensions of $(\frac{1}{2}\mathbf{D}, H^1(\mathbb{R}))$. 
However this seems trivial because we thought at first that its regular Dirichlet extension should be irreducible. It is well known that an irreducible one-dimensional diffusion process can be characterized by a scale function, a speed measure and a killing measure (Cf. \cite{IM74}). An irreducible one-dimensional diffusion must be symmetric with respect to the speed measure and its Dirichlet form has representation given in \cite[Theorem~3.1]{FHY10}. Then applying \cite[Theorem~4.1]{FHY10}, we conclude that if an (irreducible) regular Dirichlet form is a Dirichlet extension of one-dimensional Brownian motion if and only if the scale function of its associated diffusion belongs to the following class:
\begin{equation}\label{EQ1TRT}
\begin{aligned}
\mathbf{T}(\mathbb{R})
& :=\left\{\tt:\mathbb{R}\rightarrow \mathbb{R}\,|\, \text{strictly increasing and continuous, }dx\ll d\tt, \frac{dx}{d\tt}=1\text{ or }0, d\tt\text{-a.e.}\right\} \\
&=\left\{\tt: \mathbb{R}\rightarrow \mathbb{R}\,|\, \ss:=\tt^{-1}\in \mathbf{S}(\mathbb{R})\right\} \quad \text{(Cf. \cite[Theorem~4.1]{FFY05})}.
\end{aligned}
\end{equation}
Note that $\TT(\mathbb{R})\neq \{\tt=\ss^{-1}| \ss\in \SS(\mathbb{R})\}$ since the range $\ss(\mathbb{R})$ of $\ss$ may be a proper subset of $\mathbb{R}$ for some $\ss\in \SS(\mathbb{R})$ (such as the example at the end of \cite{FFY05}). At least we have proper examples, such as Example~\ref{EXA215}, for the problem \textbf{(Q.1)}.

Note that an irreducible diffusion process above is called `regular' in the terminology of \cite[(45.2)]{RW87}:
\[
	\mathbf{P}_x(\sigma_y<\infty)>0,\quad \forall x,y \in\mathbb{R},
\]
where $\mathbf{P}_x,\ {x\in \mathbb{R}}$ is the probability measure to describe the diffusion process $(X_t)_{t\geq 0}$
starting at $x$ and $\sigma_y$ the first hitting time of $\{y\}$, i.e. $\sigma_y:=\inf\{t>0: X_t=y\}$.
 When dealing with a regular Dirichlet subspace, since one-dimensional Brownian motion is irreducible,
 it follows from Proposition~\ref{PRO23}~(3) that any regular Dirichlet subspace is also irreducible,
so that the characterization of regular Dirichlet subspaces of one-dimensional Brownian motion has been completed in \cite{FFY05}.

Actually we realized that the characterization problem of regular Dirichlet extensions of one-dimensional Brownian motion 
was far from being solved when we found the following example of regular Dirichlet extension for Brownian motion which is not irreducible.
This example was an surprise for us indeed and initiated this article. 

\begin{example} Let a linear diffusion process $X$ on $\mathbb{R}$, having Lebesgue measure as speed measure, consist of two irreducible parts: a reflected Brownian motion on $I_1:=(-\infty, 0]$ and a linear diffusion on $I_2:=(0,\infty)$ with scale function $\tt$ where the range of $\tt$ is $\mathbb{R}$ and $\tt$ satisfies  that $dx\ll d\tt$ and ${dx/d\tt}=0$ or $1$. The existence of $\tt$ will be explained later. Referring to \cite{FHY10}, the Dirichlet form of $X$ on $L^2(\mathbb{R})$ is given by
\[
\begin{aligned}
 &\FF=\{f\in L^2(\mathbb{R}): f|_{I_1}\in H^1(I_1), f|_{I_2}\in H^1(I_2,d\tt)\};\\
&\EE(f,f)=\frac{1}{2}\int_{I_1} \left({\frac{df}{dx}}\right)^2dx+\frac{1}{2}\int_{I_2}\left(\frac{df}{d\tt}\right)^2 d\tt,
\end{aligned}\]
where $H^1(I_2,d\tt)=\{f\in L^2(I_2): f\ll \tt, {df/d\tt}\in L^2(I_2;d\tt)\}$.
It is easy to check that $(\EE,\FF)$ is an Dirichlet extension of $(\frac{1}{2}\mathbf{D}, H^1(\mathbb{R}))$. We need only to verify that it is regular, or precisely $\FF\cap C_c(\mathbb{R})$ is dense in $\FF$. It amounts to prove that for a function $f$ on $\mathbb{R}$ with $f|_{I_1}\in C_c^\infty(I_1)$ and $f|_{I_2}\in H^1(I_2,d\tt)\cap C_c(I_2)$, and any $\epsilon>0$, there exists $f_{\epsilon}\in \FF\cap C_c(\mathbb{R})$ such that
$$\EE_1(f-f_{\epsilon},f-f_{\epsilon})<2\epsilon.$$

We would like to spend a few lines to explain the proof because the idea inspires this paper. For simplicity we assume that $f(0)=1$. We may let $\epsilon$ small enough such that $f(\epsilon)=0$. Since $\tt(0+)=-\infty$, we may have $\epsilon'\in (0,\epsilon)$ so that $\tt(\epsilon)-\tt(\epsilon')>2/\epsilon$. Define $\varphi\in C(\mathbb{R})$
$$\varphi(x):=\begin{cases}1,& x\le t(\epsilon');\\
\frac{\tt(\epsilon)-x}{\tt(\epsilon)- \tt(\epsilon')},& x\in {(\tt(\epsilon'),\tt(\epsilon))},\\
0,& x\ge \tt(\epsilon),\end{cases}$$
and $$f_{\epsilon} :=f\cdot 1_{\mathbb{R}\backslash (0,\epsilon]}+\varphi\circ \tt\cdot1_{(0,\epsilon]}.$$
Then $f_{\epsilon}\in C_c(\mathbb{R})$, $f_\epsilon-f=\varphi\circ \tt$ and
\begin{align*} \EE_1(\varphi\circ\tt,\varphi\circ\tt)&=\int_0^\epsilon (\varphi\circ\tt(x))^2dx+
\frac{1}{2} \int_0^\epsilon \left(\frac{d\varphi\circ\tt}{d\tt}\right)^2d\tt\\
&\le \epsilon+\frac{1}{2}\int_{\tt(\epsilon')}^{\tt(\epsilon)} (\varphi'(x))^2dx\le 2\epsilon.\end{align*}
\end{example}

From this example, we know that the extensions come from two aspects: one is the singularity of scale function and the other is
the violence of irreducibility. The main purpose of this article is to give a complete characterization of extensions for one-dimensional Brownian motion.
After having characterization theorem, we are naturally interested in the structure of regular extensions.
In \cite{LY14}, we investigated the structure of regular Dirichlet subspace $(\EE^1,\FF^1)$ by using trace. It is evident that any scale function $\ss$ in \eqref{EQ1SRS} could induce a measure-dense set (i.e., for any $a<b$, $m((a,b)\cap G_\ss)>0$)
\[
G_\ss:=\{x: \ss'(x)=1\}
\]
and vice versa. By enforcing a basic assumption: `$G_\ss$ has an open version', we first claimed that before leaving $G_\ss$, $(\EE^1,\FF^1)$ is nothing but a Brownian motion (Cf. \cite[Lemma~2.2]{LY14}). Then their differences are focused on the traces on $G_\ss^c$ and the trace formulae are attained in \cite[Theorem~2.1]{LY14} by using the results of  \cite{CFY06}. We shall apply the same idea in this paper to analyze the structure of extension.

This paper is organized as follows. In \S\ref{SEC0}, we summarize some basic properties concerning regular Dirichlet extensions in the general setting. Particularly, a regular Dirichlet extension of one-dimensional Brownian motion must be strongly local and recurrent. Thus the associated Hunt process is a conservative diffusion process on $\mathbb{R}$. In \S\ref{SEC2}, we treat the problem \textbf{(Q.2)} for one-dimensional Brownian motion. The main theorem, i.e. Theorem~\ref{THM23}, characterizes all the regular Dirichlet extensions of one-dimensional Brownian motion. It turns out that every regular Dirichlet extension of one-dimensional Brownian motion has countable invariant intervals and on each of such intervals, the regular Dirichlet extension is determined uniquely by a scale function in the class \eqref{EQ2TIT}. Moreover, the complement of these intervals is an $\EE$-polar set relative to this regular Dirichlet extension. Several examples of proper regular Dirichlet extensions are presented in \S\ref{SEC23}.  In \S\ref{SEC3} and \S\ref{SEC5}, we consider the problem \textbf{(Q.3)} for one-dimensional Brownian motion and describe the structures of regular Dirichlet extensions via the trace method introduced in \cite{LY14}. We attain the expression of the orthogonal complement $\mathcal{G}$ of $H^1_\mathrm{e}(\mathbf{R})$ in $\FF_\mathrm{e}$ in Theorem~\ref{THM32} and the regular representation of the restriction of $\mathcal{G}$ on each invariant interval via the darning method in Theorem~\ref{THM313}. The darning process turns out to be a Brownian motion being time changed by a Radon smooth measure. The trace formulae of regular Dirichlet extension and the one-dimensional Brownian motion are formulated in Theorem~\ref{THM38}. In Corollary~\ref{COR39}, a special case of Theorem~\ref{THM38} is emphasized, in which the trace Dirichlet forms of one-dimensional Brownian motion and its regular Dirichlet extension are both pure-jump type and have the same jumping measure but different Dirichlet spaces. The essential differences between them are illustrated in Corollary~\ref{COR311}. Roughly speaking, the trace of Brownian motion is irreducible, whereas the trace of regular Dirichlet extension is not irreducible.

\subsubsection*{Notations and terminologies}
Let us put some often used notations here for handy reference, though we may restate their definitions when they appear.

For $a<b$, $\langle a, b\rangle$ is an interval where $a$ or $b$ may or may not be contained $\langle a, b\rangle$.
Notations $m$, $dx$ and $|\cdot|$ stand for the Lebesgue measure on $\mathbb{R}$ throughout the paper if no confusion caused.%
The restrictions of a measure $\mu$ and a function $f$ on $I$ are denoted by $\mu|_I$ and $f|_I$ respectively.
The notation `$:=$' is read as `to be defined as'.

For a scale function $\tt$ (i.e. a continuous and strictly increasing function) on some interval $I$, $d\tt$ represents its associated Lebesgue-Stieltjes measure on $I$. Set $\tt(I):=\{\tt(x):x\in I\}$. For two measures $\mu$ and $\nu$, $\mu\ll \nu$ means $\mu$ is absolutely continuous with respect to $\nu$. Given a scale function $\tt$ on $I$ and another function $f$ on $I$, $f\ll \tt$ means $f=g\circ \tt$ for some absolutely continuous function $g$ and $\frac{df}{d\tt}:=g'\circ \tt$. Given an interval $I$,  the classes $C_c(I), C^1_c(I)$ and $C^\infty_c(I)$ denote the spaces of all continuous functions on $I$ with compact support, all continuously differentiable functions with compact support and all infinitely differentiable functions with compact support, respectively.

For a Markov process $X=(X_t)_{t\geq 0}$ associated with a Dirichlet form $(\EE,\FF)$ on $L^2(E,m)$, $(P_t)_{t\geq 0}$ represents its probability transition semigroup, i.e. $P_tf(x):=\mathbf{E}_xf(X_t)$ for any $t\geq 0, f\in b\mathcal{B}(E)$ and $x\in E$, where $b\mathcal{B}(E)$ is all bounded Borel measurable functions on $E$. The semigroup $(T_t)_{t\geq 0}$ is a strongly continuous contraction semigroup on $L^2(E,m)$ associated with $(\EE,\FF)$. If $A$ is an invariant set of $X$ (see \S\ref{SEC21}), then the restriction of $(\EE,\FF)$ to $A$ is denoted by $(\EE^A,\FF^A)$ and the restriction of $X$ to $A$ is denoted by $X^A$. If $U$ is an open subset of $E$, then the part Dirichlet form of $(\EE,\FF)$ on $U$ is denoted by $(\EE_U,\FF_U)$ and the part process of $X$ on $U$ is denoted by $X_U$.
All terminologies about Dirichlet forms are standard and we refer them to \cite{CF12, FOT11}.

\section{Basic properties of regular Dirichlet extensions}\label{SEC0}

In this section, we summarize several basic properties of regular Dirichlet extensions or subspaces, which are contained in \cite{LS16, LS16-2, LUY15, LY15, LY15-3, LY15-2}. We always fix two regular Dirichlet forms $(\EE^1,\FF^1)$ and $(\EE^2,\FF^2)$ on $L^2(E,m)$ and assume that $(\EE^1,\FF^1)$ is a regular Dirichlet subspace of $(\EE^2,\FF^2)$, equivalently $(\EE^2,\FF^2)$ is a regular Dirichlet extension of $(\EE^1,\FF^1)$.

The first theorem is taken from \cite{LY15}, and it characterizes Beurling-Deny decompositions of regular Dirichlet subspaces or extensions.

\begin{theorem}[Theorem~2.1, \cite{LY15}]\label{THM21}
 Let $(J_1, k_1)$ and $(J_2,k_2)$ be the jumping and killing measures in the Beurling-Deny decompositions of $(\EE^1,\FF^1)$ and $(\EE^2,\FF^2)$ respectively. Then $J_1=J_2$ and $k_1=k_2$.
\end{theorem}

 As a corollary of this result, if one of $(\EE^1,\FF^1)$ and $(\EE^2, \FF^2)$ is strongly local or local, then the other one has to be strongly local or local. Particularly, both regular Dirichlet subspaces and extensions of $(\frac{1}{2}\mathbf{D}, H^1(\mathbb{R}))$ must be a strongly local Dirichlet form.

The following proposition describes the quasi notions of $(\EE^1,\FF^1)$ and $(\EE^2,\FF^2)$. Its proof is obvious from the fact $\text{Cap}^1(A)\geq \text{Cap}^2(A)$ for any appropriate set $A$, where $\text{Cap}^1$ and $\text{Cap}^2$ are the $1$-Capacities of $(\EE^1,\FF^1)$ and $(\EE^2,\FF^2)$ respectively.

\begin{proposition}[Remark~1.1, \cite{LY15}]\label{PRO22}
The following assertions hold.
\begin{itemize}
\item[(1)] An $\EE^1$-polar set is $\EE^2$-polar.
\item[(2)] An $\EE^1$-nest is also an $\EE^2$-nest.
\item[(3)] An $\EE^1$-quasi continuous function is also $\EE^2$-quasi continuous.
\end{itemize}
\end{proposition}

Another proposition states the relation of their global properties. 

\begin{proposition}[Remark~3.5, \cite{LY15-3}]\label{PRO23}
The following assertions hold.
\begin{itemize}
\item[(1)] If a Dirichlet form is transient, then its regular Dirichlet subspace is also  transient.
\item[(2)] If a Dirichlet form is recurrent, then its regular Dirichlet extension is also recurrent.
\item[(3)] If a Dirichlet form is strongly local and irreducible, then its regular Dirichlet subspace is also irreducible.
\end{itemize}
\end{proposition}
\begin{proof}
The first and second assertions are the direct corollaries of \cite[Theorem~1.6.4]{FOT11}. The third assertion follows from Proposition~\ref{PRO22}~(3) and \cite[Theorem~4.6.4]{FOT11}.
\end{proof}

The following characterization via the extended Dirichlet spaces is very simple but sometimes very useful.

\begin{proposition}
Let $\FF^1_\mathrm{e}$ and $\FF^2_\mathrm{e}$ be the extended Dirichlet spaces of $(\EE^1,\FF^1)$ and $(\EE^2,\FF^2)$ respectively. Then $(\EE^1,\FF^1)$ is a regular Dirichlet subspace of $(\EE^2,\FF^2)$ if and only if
\[
	\FF^1_\mathrm{e}\subset \FF^2_\mathrm{e},\quad \EE^2(f,g)=\EE^1(f,g),\quad f,g\in \FF^1_\mathrm{e}.
\]
Furthermore, if $(\EE^1,\FF^1)$ is a proper one in addition, $\FF^1_\mathrm{e}\neq \FF^2_\mathrm{e}$.
\end{proposition}

The next proposition will be frequently used in \S\ref{SEC224}. The proof is direct from the definition of part Dirichlet form  (Cf. \cite[\S4.4]{FOT11}).

\begin{proposition}\label{PRO25}
Let $U$ be an open subset of $E$. The part Dirichlet forms of $(\EE^1,\FF^1)$ and $(\EE^2, \FF^2)$ on $U$ are denoted by $(\EE_U^1,\FF_U^1)$ and $(\EE_U^2, \FF_U^2)$. Then $(\EE^1_U,\FF^1_U)$ is a regular Dirichlet subspace of $(\EE^2_U,\FF^2_U)$.
\end{proposition}

\section{Representation of regular Dirichlet extensions}\label{SEC2}

\subsection{Main result}\label{SEC21}
The existence problem \textbf{(Q.1)} for one-dimensional Brownian motion is already answered in the next paragraph after this problem in \S\ref{SEC1}. Indeed, the one-dimensional Brownian motion has proper regular Dirichlet extensions such as those with the scale functions in the class \eqref{EQ1TRT}. Particularly, they are all irreducible. In this section, we shall treat the second problem \textbf{(Q.2)}.

Before presenting the main theorem, we need to do some preparatory works. Let $(\EE,\FF)$ be a regular Dirichlet form on $L^2(E,m)$ associated with a symmetric Hunt process $X$. A Borel subset $A\subset E$ is called an invariant set of $(X_t)_{t\geq 0}$ provided for any $x\in A$,
\[
\mathbf{P}_x(X_t\in A, \forall t)=1.
\]
Clearly, the restriction denoted by $X^A$ or $(X^A_t)_{t\geq 0}$ of $X$ to $A$ is still a Hunt process and symmetric with respect to $m|_A$. Its associated Dirichlet form on $L^2(A, m|_A)$ is (see \cite[\S2.1]{CF12})
\[
\begin{aligned}
\FF^A:=\{f|_A: f\in \FF\},\quad \EE^A(f|_A,g|_A):=\EE(1_Af,1_Ag),\quad f,g\in \FF.
\end{aligned}
\]
We call $(\EE^A, \FF^A)$ the restriction of the Dirichlet form $(\EE,\FF)$ to the invariant set $A$.

Another preparatory work is to introduce a few classes of scale functions. Let $a<b$ and $I=\langle a, b\rangle$ be an interval such that $a$ or $b$ may or may not be in $I$. In other words, $I=(a,b), (a,b], [a,b)$ or $[a,b]$. Particularly, $a$ or $b$ may be infinity if $a$ or $b\notin I$. The interior of $I$ is denoted by $\overset{\circ}{I}:=(a,b)$. A scale function $\tt$ on $I$ is a strictly increasing and continuous function on $I$. Thus we can always define its limit at boundary
\[
	\tt(a):=\lim_{x\downarrow a}\tt(x)\geq -\infty,\quad (\text{resp. }\tt(b):=\lim_{x\uparrow b} \tt(x)\leq \infty)
\]
no matter $a$ (resp. $b$) belongs to $I$ or not. Denote all the scale functions $\tt$ on $I$ satisfying
\[
	dx\ll d\tt,\quad  \frac{dx}{d\tt}=0\text{ or }1, \quad d\tt\text{-a.e.}
\]
by $\TT(I)$ (see \eqref{EQ1TRT}). A subset of $\TT(I)$ is defined as
\begin{equation}
\TT_\infty(I):=\{\tt\in \TT(I)\, |\, \tt(a)=-\infty\text{ iff }a\notin I, \tt(b)=\infty\text{ iff }b\notin I\},
\end{equation}
where `iff' stands for `if and only if'.

\begin{remark}
Note that in any case the class $\TT_\infty(I)$ of scale functions is not empty.
 For example let us treat the case $I=[a,b)$ with $b<\infty$. The other cases can be treated similarly. By \cite[Theorem~4.1]{FFY05}, we need only to find a scale function
\[
	\ss: [0, \infty) \rightarrow [a,b)
\]
such that $d\ss\ll dx, \ss'=0$ or $1$. Then $\tt:=\ss^{-1}\in \TT_{\infty}(I)$.

Take a measure-dense subset $G\subset [0,\infty)$. For example, assume $\{q_n: n\geq 1\}$ is the total of positive rational numbers and let
\[
	G:=\left\{\bigcup_{n\geq 1} B\left(q_n, \frac{1}{2^n}\right)\right\} \cap [0,\infty),
\]
where $B(x,r):=\{y: |y-x|<r\}$. Clearly, the Lebesgue measure of $G$ is positive, i.e. $|G|>0$. Set $k:=|G|/|b-a|$ and $G':=\{x: kx\in G\}$. Note that $G'$ is still measure-dense. In fact, take any open interval $(c,d)\subset I$, we have
\[
|G'\cap (c,d)|= \frac{1}{k}|G\cap (kc, kd)|>0.
\]
Let
\[
	\ss(x):=\int_0^x 1_{G'}(y)dy+a,\quad x\geq 0.
\]
Then $\ss$ is strictly increasing and absolutely continuous, $\ss'=1_{G'}$, $\ss(0)=a$ and
\[
	\ss(\infty)=\int_0^\infty 1_{G'}(y)dy+a=\frac{1}{k}\cdot |G|+a=b.
\]
\end{remark}

\begin{remark}
Similar to \cite[Theorem~4.1]{FFY05}, we may also deduce that any scale function $\tt\in \TT_\infty(I)$ can be written as
\[
\tt(x)=x+c(x)
\]
for a non-decreasing singular continuous function $c$ on $I$.
\end{remark}

Note that the scale functions of an irreducible diffusion process are not unique and may differ by a constant if its speed measure is fixed.
To avoid this uncertainty, we make the following restriction on $\TT_\infty(I)$:
\begin{equation}\label{EQ2TIT}
\TT^0_\infty(I):=\left\{\tt\in \TT_\infty(I): \tt\left(e\right)=0\right\},
\end{equation}
where $e$ is a fixed point in $(a,b)$: $e=(a+b)/2$ if $a>-\infty, b<\infty$, $e=b-1$ if $a=-\infty, b<\infty$, $e=a+1$ if $a>-\infty, b=\infty$ and $e=0$ if $a=-\infty, b=\infty$. The choice of $e$ is not essential.

Now we are in a position to state the main result of this section.
Note that $m$ represents the Lebesgue measure on $\mathbb{R}$ in  the following theorem.

\begin{theorem}\label{THM23}
The Dirichlet form $(\EE,\FF)$ is a regular Dirichlet extension of $(\frac{1}{2}\mathbf{D},H^1(\mathbb{R}))$ on $L^2(\mathbb{R})$ if and only if
there exist a set of at most countable disjoint intervals $\{I_n=\langle a_n, b_n\rangle: n\geq 1\}$,
satisfying that $ \left(\bigcup_{n\geq 1} I_n\right)^c$ 
has Lebesgue measure zero, and a scale function $\tt_n\in \TT^0_\infty(I_n)$ for each $ n\geq 1$ such that
\begin{equation}\label{EQ2FFL}
\begin{aligned}
&\FF=\left\{f\in L^2(\mathbb{R}): f|_{I_n}\in \FF^n, \sum_{n\geq 1}\EE^n(f|_{I_n}, f|_{I_n})<\infty\right\},  \\
&\EE(f,g)=\sum_{n\geq 1}\EE^n(f|_{I_n}, g|_{I_n}),\quad f,g\in \FF,
\end{aligned}\end{equation}
where for each $n\geq 1$, $(\EE^n,\FF^n)$ is expressed as
\begin{equation}\label{EQ2FNU}
\begin{aligned}
&\FF^n=\left\{f\in L^2(I_n): f\ll \tt_n, \int_{I_n}\left(\frac{df}{d\tt_n}\right)^2d\tt_n<\infty \right\}, \\
&\EE^n(f,g)=\frac{1}{2}\int_{I_n}\frac{df}{d\tt_n}\frac{dg}{d\tt_n}d\tt_n,\quad f,g\in \FF^n.
\end{aligned}
\end{equation}
Moreover, the intervals $\{I_n: n\geq 1\}$ and scale functions $\{\tt_n\in \mathbf{T}^0_\infty(I_n):n\geq 1\}$
are uniquely determined, if the difference of order is ignored.
\end{theorem}


\begin{remark}\label{RM24}
{Denote the associated Hunt process of $(\EE,\FF)$ by $X=(X_t)_{t\geq 0}$. Set
 $G:=\bigcup_{n\geq 1}\overset{\circ}{I_n}$ and $F:=G^c$. Note that $G$ is an open set.
We would like to make a few remarks for the theorem above.
\begin{itemize}
\item[(1)] Though the intervals are mutually disjoint, they may have common endpoints. For example, $I_n=(a_n, b_n]$ and $I_m=(a_m,b_m)$ with $b_n=a_m$.
\item[(2)] Let $\Lambda_{pr}:=\{a_n: a_n\in I_n\}$ and $\Lambda_{pl}:=\{b_n: b_n\in I_n\}$. Further set $\Lambda_r:=F\setminus \Lambda_{pl}$ and $\Lambda_l:=F\setminus \Lambda_{pr}$. Note that neither $\Lambda_l$ nor $\Lambda_r$ is necessarily closed. For example, let $K$ be the standard Cantor set in $[0,1]$ and set
\[
	\bigcup_nI_n := K^c \cup (-\infty, 0]\cup [1,\infty).
\]
Then $\Lambda_r=K\setminus \{0\}$ and $\Lambda_l=K\setminus \{1\}$. Neither of them is closed. Nevertheless, $\Lambda_l$ (resp. $\Lambda_r$) is closed from the right (resp. left), i.e. if $x_n\in \Lambda_l$ (resp. $\Lambda_r$) and $x_n\downarrow x$ (resp. $x_n\uparrow x$), then $x\in \Lambda_l$ (resp. $\Lambda_r$). This fact can be proved as follows. Assume that $x_n\in \Lambda_l$ and $x_n\downarrow x$. Clearly $x\notin G$ since $G$ is open. If $x\in \Lambda_{pr}$, then there exists an interval $I_n$ such that $I_n=[x, b_n)$ or $[x, b_n]$ with $x<b_n$. Note that $(x,b_n)\subset G$. This leads to a contradiction with $x_n\downarrow x$ and $x_n\in \Lambda_l$. The sets $\Lambda_{pr}, \Lambda_{pl}, \Lambda_r, \Lambda_l$ are called the classes of right shunt points, left shunt points, right singular points and left singular points respectively in \cite[\S3.4]{IM74}. The open set $G$ is called the class of regular points.
\item[(3)] For each $n$, $I_n$ is an invariant set of $X$ and $X^{I_n}$ is an irreducible and recurrent diffusion process with the scale function $\tt_n$, the speed measure $m|_{I_n}$ and no killing inside (Cf. \cite[Theorem~2.2.11]{CF12}. In other words,
\[
	\mathbf{P}_x(X^{I_n}_t=y, \exists t>0)=1,\quad \forall x,y \in I_n.
\]
Furthermore, if $a_n\in I_n$, then $X^{I_n}$ is reflected at the left endpoint $a_n$. If $a_n\notin I_n$, then $X^{I_n}$ never reach it in finite time (Cf. \cite{I57} and \cite[Example~3.5.7]{CF12}). This also implies that any single point subset $\{x\}\subset I_n$ is not an $m$-polar set relative to $X$.
\item[(4)] The set $\Lambda_l\cap \Lambda_r=\left(\bigcup_{n\geq 1}I_n\right)^c$ is an $m$-polar set relative to $X$. Indeed, $m(\Lambda_r\cap \Lambda_l)=0$, and for any $x\notin \Lambda_r\cap \Lambda_l$, there exists an interval $I_n$ such that $x\in I_n$. Since $I_n$ is an invariant set of $X$, we can conclude
\[
	\mathbf{P}_x(\sigma_{\Lambda_r\cap \Lambda_l}<\infty)=0.
\]
Note that any regular Dirichlet form corresponds to a Hunt process uniquely up to an $m$-polar set. The most convenient way to treat the part of $X$ on $\Lambda_r\cap \Lambda_l$ is to enforce the process $X$ starting from a point $x\in \Lambda_r\cap \Lambda_l$ to stay at $x$ forever.
\item[(5)] The fact that $\left(\bigcup_{n\geq 1} I_n\right)^c$ has Lebesgue measure zero implies that it is nowhere dense.
\end{itemize}}
\end{remark}

\begin{corollary}
An irreducible regular Dirichlet form $(\EE,\FF)$ on $L^2(\mathbb{R})$ is a regular Dirichlet extension of $(\frac{1}{2}\mathbf{D},H^1(\mathbb{R}))$ if and only if there exists a unique scale function $\tt\in \TT_\infty^0(\mathbb{R})$ such that
\[
\begin{aligned}
	&\FF=\left\{f\in L^2(\mathbb{R}): f\ll \tt, \int_{\mathbb{R}}\left(\frac{df}{d\tt}\right)^2d\tt<\infty \right\}, \\
&\EE(f,g)=\frac{1}{2}\int_{\mathbb{R}}\frac{df}{d\tt}\frac{dg}{d\tt}d\tt,\quad f,g\in \FF.
\end{aligned}
\]
\end{corollary}

\subsection{Proof of Theorem~\ref{THM23}}\label{SEC22}

The proof of Theorem~\ref{THM23} will be divided into several parts. We note here that the last assertion about the uniqueness is obvious from Remark~\ref{RM24}~(3). We shall prove necessity first and then sufficiency. To prove the necessity, we need to review one-dimensional or linear diffusions.

\subsubsection{Classification of points for one-dimensional diffusions}\label{SEC221}

In this part, we recall some results on the classification of points for linear diffusion. For those results which may be known to experts but not on standard references \cite[\S5]{I57} and \cite[\S3]{IM74}, we will give a proof.

Let $X=(X_t)$ be a diffusion process on $\mathbb{R}$, i.e. a strong Markov process with continuous paths. Without loss of generality, we always assume that $X$ is conservative, in other words, the lifetime $\zeta$ of $X$ is infinite $\mathbf{P}_x$-a.s. for any $x\in \mathbb{R}$. Now fix a point $x\in \mathbb{R}$. Note that
\[
	e^{\pm}:=\mathbf{P}_x(\sigma_{x\pm}=0)=0\text{ or }1
\]
by Blumenthal's $0$-$1$ law, where $\sigma_{x+}:=\inf\{t>0: X_t>x\},\ \sigma_{x-}:=\inf\{t>0: X_t<x\}$.

\begin{definition} A point $x\in\mathbb{R}$ is called
\begin{itemize}
\item[(1)] regular ($x\in \Lambda_2$), if $e^+=e^-=1$;
\item[(2)] singular ($x\in \Lambda_r\cup \Lambda_l$), if $e^+e^-=0$;
\item[(3)] left singular ($x\in \Lambda_l$), if $e^+=0$; right singular ($x\in \Lambda_r$), if $e^-=0$;
\item[(4)] left shunt ($x\in \Lambda_{pl}$), if $e^+=0, e^-=1$; right shunt ($x\in \Lambda_{pr}$), if $e^-=0, e^+=1$;
\item[(5)] a trap ($x\in \Lambda_t$), if $e^+=e^-=0$.
\end{itemize}
\end{definition}

Clearly, $\Lambda_2=\left(\Lambda_r\cup \Lambda_l\right)^c$, $\Lambda_{pr}\cap \Lambda_{l}=\emptyset$, $\Lambda_{pl}\cap \Lambda_{r}=\emptyset$ and $\Lambda_r\cap \Lambda_l=\Lambda_t$. The following facts will be very useful in proving Theorem~\ref{THM23}.

\begin{lemma}\label{LM25}
\begin{itemize}
\item[(1)] Assume $a<b<c$. Then
\[
\begin{aligned}	&\mathbf{P}_a(\sigma_c<\infty)=\mathbf{P}_a(\sigma_b<\infty)\mathbf{P}_b(\sigma_c<\infty),\\
&\mathbf{P}_c(\sigma_a<\infty)=\mathbf{P}_c(\sigma_b<\infty)\mathbf{P}_b(\sigma_a<\infty).
\end{aligned}\]
\item[(2)] A point $b\in \Lambda_r$ (resp. $b\in \Lambda_l$) if and only if $\mathbf{P}_b(X_t\geq b, \forall t)=1$ (resp.  $\mathbf{P}_b(X_t\leq b, \forall t)=1$). Thus $b\in \Lambda_t$ if and only if $\mathbf{P}_b(X_t= b, \forall t)=1$.
\item[(3)] Fix $b\in \Lambda_r$ (resp. $b\in \Lambda_l$). Then for any $a>b$ (resp. $a<b$),
\[
	\mathbf{P}_a(X_t\geq b, \forall t)=1,\quad (\text{resp. }\mathbf{P}_a(X_t\leq b, \forall t)=1).
\]
\item[(4)] Fix $b\in \Lambda_{pr}$ (resp. $b\in \Lambda_{pl}$). Then there exists a point $a>b$ (resp. $a<b$) such that
\[
\mathbf{P}_b(\sigma_a<\infty)>0.
\]
\item[(5)] The left singular set $\Lambda_l$ is closed from the right, i.e. if $x_n\in \Lambda_l$ and $x_n\downarrow x$, $x\in \Lambda_l$. The right singular set $\Lambda_r$ is closed from the left, i.e. if $x_n\in \Lambda_r$ and $x_n\uparrow x$, $x\in \Lambda_r$.
\item[(6)] The regular set $\Lambda_2$ is open. Thus the singular set $\Lambda_r\cup \Lambda_l$ is closed.
\item[(7)] If each point in an open interval $(a,b)$ is regular, i.e. $(a,b)\subset \Lambda_2$, then for any $x,y\in (a,b)$,
\[
	\mathbf{P}_x(\sigma_y<\infty)\mathbf{P}_y(\sigma_x<\infty)>0.
\]
\end{itemize}
\end{lemma}
\begin{proof}
For the first fact, since in the sense of $\mathbf{P}_a$-a.s., $\sigma_c>\sigma_b$, it follows that $\sigma_c=\sigma_b+\sigma_c\circ \theta_{\sigma_b}$ where $(\theta_t)$ are the shift operators of $X$, i.e. $X_{t+s}=X_t\circ \theta_s$ for any $t,s\geq 0$. By the strong Markovian property of $X$, we have
\[
\begin{aligned}
\mathbf{P}_a(\sigma_c<\infty)&= \mathbf{P}_a(\sigma_b<\infty, \sigma_c\circ \theta_{\sigma_b}<\infty)\\
&=\mathbf{P}_a(\sigma_b<\infty, \mathbf{P}_{X_{\sigma_b}}(\sigma_c<\infty))\\
&=\mathbf{P}_a(\sigma_b<\infty)\mathbf{P}_b(\sigma_c<\infty).
\end{aligned}
\]
Another assertion can be deduced similarly.

For the second fact, we need only to remark that $\mathbf{P}_x(\sigma_{x\pm}=0)=\mathbf{P}_x(\sigma_{x\pm}<\infty)$ for any $x\in \mathbb{R}$ (Cf. \cite[\S3.3, 10a)]{IM74}).

For the third fact, fix $b\in \Lambda_r$ and $a>b$. For any point $y<b$, it follows from (2) that
\[
\mathbf{P}_b(\sigma_y<\infty)=\mathbf{P}_b(X_t=y,\exists t)=0.
\]
Thus from (1) we may deduce that $\mathbf{P}_a(\sigma_y<\infty)=0$ for any $y<b$. Take a sequence $y_n\uparrow b$. Then $\mathbf{P}_a(\sigma_{y_n}<\infty)=0$ implies
\[
	\mathbf{P}_a\left(\bigcup_{t} \{X_t\leq y_n\}\right)=0.
\]
Hence
\[
	0=\mathbf{P}_a\left(\bigcup_n\bigcup_{t} \{X_t\leq y_n\}\right)=\mathbf{P}_a\left(\bigcup_t\{X_t<b\}\right)
	=\mathbf{P}_a\left(\{X_t\geq b, \forall t\}^c\right).
\]
Another assertion is similar.

For the forth fact, fix $b\in \Lambda_{pr}$. Suppose that for any $y>b$, $$\mathbf{P}_b(\sigma_y<\infty)=\mathbf{P}_b\left(\cup_t\{X_t\geq y\}\right)=0.$$ Take a sequence $y_n\downarrow b$ and then
\[
0=\mathbf{P}_b\left(\bigcup_n\bigcup_t\{X_t\geq y_n\}\right)=\mathbf{P}_b\left(\bigcup_t\{X_t>b\}\right).
\]
This implies $\mathbf{P}_b(X_t\leq b,\forall t)=1$ and thus $b\in \Lambda_l$ by (2), which contradicts with $\Lambda_{pr}\cap \Lambda_l=\emptyset$.

The fifth and sixth facts can be found in \cite[\S3.4]{IM74}.

For the final fact, note that for any regular point $\xi$, there exist two points $c,d$ close enough to $\xi$ such that $c<\xi<d$ and $\mathbf{P}_c(\sigma_d<\infty)\mathbf{P}_d(\sigma_c<\infty)>0$ (Cf. \cite[\S3.4]{IM74}). Now fix a regular interval $(a,b)\subset \Lambda_2$ and assume that $x,y\in (a,b)$, $x<y$ and $\mathbf{P}_x(\sigma_y<\infty)=0$. Set
\[
A_x:=\left\{z>x: \mathbf{P}_x(\sigma_z<\infty)=0\right\}.
\]
Clearly, $y\in A_x$. Moreover, if $z\in A_x$ and $z'>z$ then $z'\in A_x$. Let $w:=\inf A_x$. If $w=x$, then $\mathbf{P}_x(\sigma_{x+}<\infty)=0$ and $x\in \Lambda_l$, which contradicts with $x\in \Lambda_2$. If $w>x$, note that $w\in (a,b)$ is a regular point. It follows that there exist two points $w_1, w_2$ with $x<w_1<w<w_2<y$ such that
\[
\mathbf{P}_{w_1}(\sigma_{w_2}<\infty)\mathbf{P}_{w_2}(\sigma_{w_1}<\infty)>0.
\]
Since $w_1\notin A_x$ and $w_2\in A_x$, we have $\mathbf{P}_x(\sigma_{w_1}<\infty)>0$ and $\mathbf{P}_x(\sigma_{w_2}<\infty)=0$. However, from (1) we can deduce that
\[
	\mathbf{P}_x(\sigma_{w_2}<\infty)=\mathbf{P}_x(\sigma_{w_1}<\infty)\mathbf{P}_{w_1}(\sigma_{w_2}<\infty)>0,
\]
which leads to a contradiction. That completes the proof.
\end{proof}

{Intuitively, a left (resp. right) singular point looks like a `wall' to the left (resp. right), and no trajectory can run through it from its left (resp. right) side to the right (resp. left). The left (resp. right) shunt point means more: the trajectories starting from this point must enter its left (resp. right) side in finite time.

We need to point out $X$ admits a left or right shunt interval $(a,b)$, i.e. $(a,b)\subset \Lambda_{pr}$ or $\Lambda_{pl}$. For example, let $X_t=X_0+t$. Then $\Lambda_{pr}=\mathbb{R}$. This example also indicates that for a right shunt point $b$, there may exist another point $a<b$ such that the trajectory starting from $a$ can run through $b$ to its right side. We shall see in the next part that these behaviors are not allowed under the symmetry assumption. }

\subsubsection{Linear diffusion under the symmetry}\label{SEC222}

In this part, we further assume that $X$ is symmetric with respect to a fully supported Radon measure $m$ on $\mathbb{R}$. In other words, the semigroup $(P_t)_{t\geq 0}$ of $X$ satisfies
\begin{equation}\label{EQ2PTF}
	(P_tf,g)_m=(f,P_tg)_m,\quad f,g\in \mathcal{B}_b(\mathbb{R})\cap L^2(\mathbb{R},m), t\geq 0,
\end{equation}
where $(f,g)_m$ and $\mathcal{B}_b(\mathbb{R})$ stand for the inner product of $L^2(\mathbb{R},m)$ and the set of all the bounded Borel measurable functions on $\mathbb{R}$, respectively.

\begin{lemma}\label{LM26}
Fix a right (resp. left) shunt point $b\in \Lambda_{pr}$ (resp. $\Lambda_{pl}$). Under the symmetry, for any $a<b$ (resp. $a>b$), it holds that
\[
	\mathbf{P}_a(\sigma_b<\infty)=0.
\]
\end{lemma}
\begin{proof}
Fix $b\in \Lambda_{pr}\subset \Lambda_r$. It follows from Lemma~\ref{LM25}~(2, 3) that for any $x\geq b$,
\[
\mathbf{P}_x(X_t\geq b, \forall t)=1.
\]
Take a constant $N$ large enough, and set $f(x):=1_{[-N,b)}(x), g(x):=1_{[b,N]}(x)$. Since for any $x\geq b$,
\[
P_tf(x)=\mathbf{P}_x\left(X_t\in [-N,b)\right)\leq \mathbf{P}_x(X_t<b)=0,
\]
the left side of \eqref{EQ2PTF} equals $0$. Thus for $m$-a.e. $x\in [-N, b)$,
\[
	0=P_tg(x)=\mathbf{P}_x(X_t\in [b, N]).
\]
By letting $N\uparrow \infty$, we obtain that for any fixed $t>0$,
\begin{equation}\label{EQ2PXX}
	\mathbf{P}_x(X_t\geq b)=0,
	\end{equation}
 for $m$-a.e. $x<b$. Thus for $m$-a.e. $x<b$, \eqref{EQ2PXX} holds for  any $t\in \mathbb{Q}\cap (0,\infty)$, where $\mathbb{Q}$ is the set of all rational numbers. Take a point $x<b$ such that \eqref{EQ2PXX} holds for any $t\in \mathbb{Q}\cap (0,\infty)$. We have
 \[
 	\mathbf{P}_x\left(\bigcup_{t\in \mathbb{Q}, t>0} \{X_t\geq b\}\right)=0.
 \]
It follows that
\[
\mathbf{P}_x\left(\bigcap_{t\in \mathbb{Q}, t>0}\{X_t<b\}\right)=\mathbf{P}_x\left(\left\{\bigcup_{t\in \mathbb{Q}, t>0}\{X_t\geq b\}\right\}^c  \right)=1.
\]
Since $X$ is continuous, we may conclude that
\begin{equation}\label{EQ2PXT}
\mathbf{P}_x\left(\bigcap_{t}\{X_t\leq b\}\right)=1.
\end{equation}
As a result,
\begin{equation}\label{EQ2PXY}
	\mathbf{P}_x(\sigma_y<\infty)=0
\end{equation}
 for any $y>b$. Note that $m$ has full support and thus we may take a sequence $x_n \uparrow b$ such that \eqref{EQ2PXY} holds for $x=x_n$. For any $z<x_n$, it follows from Lemma~\ref{LM26}~(1) that $\mathbf{P}_z(\sigma_y<\infty)=0$ for any $y>b$. Hence $\mathbf{P}_z(\sigma_y<\infty)=0$ for any $z<b<y$.
 Therefore, from Lemma~\ref{LM25}~(1) and (4), we assert that $\mathbf{P}_x(\sigma_b<\infty)=0$ for any $x<b$. That completes the proof.
\end{proof}


The following lemma indicates that $X$ is non-decreasing (resp. non-increasing) in the right (resp. left) singular interval. However, if $X$ is symmetric, then any point in a right or left singular interval must be a trap.

\begin{lemma}\label{LM27}
\begin{itemize}
\item[(1)] If an open interval $(a,b)\subset \Lambda_r$ (resp. $\Lambda_l$), then for any $x\in (a,b)$,
\[
	\mathbf{P}_x(X_t\geq X_s, \forall s<t\leq \sigma_b)=1 \quad (\text{resp. }\mathbf{P}_x(X_t\leq X_s, \forall s<t\leq \sigma_a)=1).
\]
\item[(2)] Under the symmetry, if $(a,b)\subset \Lambda_r$ (resp. $\Lambda_l$), then $(a,b)\subset \Lambda_t$.
\end{itemize}
\end{lemma}
\begin{proof}
We first prove (1) and only consider the case $(a,b)\subset \Lambda_r$. Since any point $x\in (a,b)$ is right singular, we have
\[
\mathbf{P}_x(X_t\geq X_0,\forall t)=1
\]
 by Lemma~\ref{LM25}~(2). From the Markovian property of $X$, we can deduce that for fixed $s<t$,
\[
\begin{aligned}
\mathbf{P}_x(X_t<X_s, s<t\leq \sigma_b) &=\mathbf{P}_x\left(\{X_{t-s}<X_0\}\circ \theta_s, t-s \leq  \sigma_b\circ \theta_s, s<\sigma_b\right) \\
&=\mathbf{P}_x\left(\mathbf{P}_{X_s}(X_{t-s}<X_0, t-s\leq \sigma_b); s<\sigma_b \right)\\
&=0.
\end{aligned}
\]
The last equality above follows from the fact that, $\mathbf{P}_x$-a.s. on $\{s<\sigma_b\}$, $X_s\in (a,b)$. It is then clear that
\[
	\mathbf{P}_x\left(\bigcup_{s<t\leq \sigma_b, s,t\in \mathbb{Q}}\{X_t<X_s\} \right)=0.
\]
Thus
\[
	1=\mathbf{P}_x\left(\bigcap_{s<t\leq \sigma_b, s,t\in \mathbb{Q}}\{X_t\geq X_s\} \right)=\mathbf{P}_x\left(\bigcap_{s<t\leq \sigma_b}\{X_t\geq X_s\} \right).
\]

For the second assertion (2), fix $x\in (a,b)\subset \Lambda_r$. Take another point $w$ in $(a,b)$ such that $x<w$. Mimicking the proof of \eqref{EQ2PXT}, we deduce that
\[
	\mathbf{P}_x(\sigma_{w+}<\infty) =0.
\]
Take a sequence $w_n\downarrow x$ and we then have
\[
0=\mathbf{P}_x\left(\bigcup_n\bigcup_{t}\{X_t> w_n\}\right)=\mathbf{P}_x\left(\bigcup_{t}\{X_t> x\}\right).
\]
Thus
\[
	\mathbf{P}_x(X_t\leq x, \forall t)=1.
\]
It follows from Lemma~\ref{LM25}~(2) that $x\in \Lambda_l$ and then $x\in \Lambda_r\cap \Lambda_l=\Lambda_t$. It concludes that $(a,b)\subset \Lambda_t$.
\end{proof}

\subsubsection{A merging theorem}\label{SEC223}

Before proving Theorem~\ref{THM23}, we need a result to merge a sequence of Dirichlet forms into a new one.
Because it holds in general and may have independent interest, we state it as a theorem.
\begin{theorem}\label{LM28}
Let $E:=\cup_{n\geq 1}E_n$ with $\{E_n:n\geq 1\}$ disjoint be a measurable space and $m$ a $\sigma$-finite measure on it. Denote the restriction of $m$ to $E_n$ by $m_n$. Assume that $(\EE^n,\FF^n)$ is a Dirichlet form on $L^2(E_n,m_n)$. Then
\begin{equation}\label{EQ2FFLE}
\begin{aligned}
	&\FF:=\{f\in L^2(E,m): f|_{E_n}\in \FF^n, \sum_{n\geq 1}\EE^n(f|_{E_n},f_{E_n})<\infty\}, \\
	&\EE(f,g):= \sum_{n\geq 1}\EE^n(f|_{E_n}, f|_{E_n}),\quad f,g\in \FF
\end{aligned}
\end{equation}
is a Dirichlet form on $L^2(E,m)$.
\end{theorem}
\begin{proof}
Let $(T^n_t)$ be the semigroup of $(\EE^n,\FF^n)$ on $L^2(E_n,m_n)$. For any $f\in L^2(E,m), t\geq 0$, define
\begin{equation}\label{EQ2TTF}
(T_tf)|_{E_n}(x):= \sum_{n\geq 1}T^n_t(f|_{E_n})(x), \quad n\geq 1.
\end{equation}
 Set $f^n:=f|_{E_n}$ for convenience. We assert that $(T_t)$ is a strongly continuous and symmetric contraction semigroup on $L^2(E,m)$. The semigroup property is clear from those of $\{(T^n_t):n\geq 1\}$. For the contraction property, fix $f\in L^2(E,m)$. The $L^2$-norm of $L^2(E_n,m_n), L^2(E,m)$ are denoted by $\|\cdot\|_{E_n}, \|\cdot\|_{E}$ for short. Note that $\|f\|_E^2=\sum_{n\geq 1}\|f^n\|_{E_n}^2$. Then we have
\[
	\|T_tf\|_{E}=\sum_{n\geq 1} \|T_t^nf^n\|^2_{E_n}\leq \sum_{n\geq 1}\|f^n\|^2_{E_n}= \|f\|_E^2.
\]
To prove strong continuity, we fix $f\in L^2(E,m)$ and $\epsilon>0$, and take an integer $n$ large enough such that $\sum_{k>n}\|f^k\|_{E_k}^2<\epsilon/4$. By the strong continuity of $\{(T^k_t): 1\leq k\leq n\}$, we may take $t_\epsilon>0$ such that for any $t<t_\epsilon$,
\[
	\|T^k_tf^k-f^k\|^2_{E_k}<\frac{\epsilon}{2^k},\quad 1\leq k\leq n.
\]
Then we have for any $t<t_\epsilon$,
\[
\begin{aligned}
\|T_tf-f\|^2_E&=\sum_{1\leq k\leq n}\|T^k_tf^k-f^k\|^2_{E_k}+\sum_{ k> n}\|T^k_tf^k-f^k\|^2_{E_k} \\
&< \sum_{1\leq k\leq n}\frac{\epsilon}{2^k}+ 4\sum_{k>n}\|f^k\|_{E_k}^2\\
&<2\epsilon.
\end{aligned}\]
Therefore, $(T_t)$ corresponds uniquely to a closed form $(\EE',\FF')$ on $L^2(E,m)$. Precisely,
\[
\begin{aligned}
	&\FF'=\left\{f\in L^2(E,m): \uparrow \lim_{t\downarrow 0} \frac{1}{t}(f-T_tf,f)_m<\infty\right\}, \\
	&\EE'(f,f)=\lim_{t\downarrow 0} \frac{1}{t}(f-T_tf,f)_m,\quad f\in \FF'.
	\end{aligned}
\]
Note that the limit above is an increasing limit as $t\downarrow 0$. On the other hand,
\[
\lim_{t\downarrow 0}\frac{1}{t}(f-T_tf,f)_m=\lim_{t\downarrow 0}\sum_{n\geq 1} \frac{1}{t}(f^n-T^n_tf^n,f^n)_{m_n}=\sum_{n\geq 1}\lim_{t\downarrow 0} \frac{1}{t}(f^n-T^n_tf^n,f^n)_{m_n}.
\]
Thus $f\in \FF'$ if and only if $f^n\in \FF^n$ and $\sum_{n\geq 1} \EE^n(f^n,f^n)<\infty$. In other words,
\[
	(\EE',\FF')=(\EE,\FF).
\]
The Markovian property of $(\EE,\FF)$ may be deduced as follows. Let $\varphi$ be a normal contraction on $\mathbb{R}$ and $f\in \FF$. Note that $(\varphi\circ f)|_{E_n}=\varphi(f|_{E_n})\in \FF^n$ and $\EE^n(\varphi(f|_{E_n}), \varphi(f|_{E_n}))\leq \EE^n(f|_{E_n},f|_{E_n})$ since $(\EE^n,\FF^n)$ satisfies the Markovian property. Hence $\varphi\circ f\in \FF$ and $\EE(\varphi\circ f, \varphi\circ f)\leq \EE(f,f)$. That completes the proof.
\end{proof}

Note that the semigroup of $(\EE,\FF)$ in Theorem~\ref{LM28} is characterized by \eqref{EQ2TTF}. From this fact, we have the following corollary.

\begin{corollary}\label{COR210}
Let $(\EE,\FF)$ be a Dirichlet form on $L^2(E,m)$ associated with a symmetric Markov process $X$. Suppose that $\{E_n:n\geq 1\}$ is a sequence of disjoint invariant sets of $X$ and
\[
	E=\bigcup_{n\geq 1}E_n \quad m\text{-a.e.}
\]
Denote $(\EE^n,\FF^n):=(\EE^{E_n}, \FF^{E_n})$. Then $(\EE,\FF)$ can be expressed as \eqref{EQ2FFLE}.
\end{corollary}

\subsubsection{Proof of necessity}\label{SEC224}

In this part, we prove the necessity of Theorem~\ref{THM23}. Note that $m$ stands for the Lebesgue measure on $\mathbb{R}$ in this part.
Let $(\EE,\FF)$ be a regular Dirichlet extension of $(\frac{1}{2}\mathbf{D}, H^1(\mathbb{R}))$ on $L^2(\mathbb{R})$ associated with a Hunt process $X$. It follows from Theorem~\ref{THM21} and Proposition~\ref{PRO23}~(2) that $(\EE,\FF)$ is strongly local and recurrent. Without loss of generality, by \cite[Theorem~4.5.1~(3)]{FOT11}, we may assume that $X$ is a recurrent (hence conservative, see \cite[Lemma~1.6.5]{FOT11}) diffusion process on $\mathbb{R}$.

We use the same notations as \S\ref{SEC221} to denote the classes of points for $X$. Let
\[
	G:=\Lambda_2
\]
be the set of regular points and which is open by Lemma~\ref{LM25}~(6). Thus $G$ may be written as a union of countable disjoint open intervals:
\begin{equation}\label{EQ2GNA}
	G=\bigcup_{n\geq 1}(a_n,b_n).
\end{equation}
We assert $F:=G^c$ is nowhere dense, and the shunt point must be an endpoint of some interval in \eqref{EQ2GNA}.

\begin{lemma}\label{LM29}
The singular set $F=\Lambda_r\cup \Lambda_l$ is nowhere dense. Furthermore, $F\setminus \{a_n,b_n: n\geq 1\}\subset \Lambda_t$.
\end{lemma}
\begin{proof}
We first prove $\Lambda_r$ has empty interior. Assume that $(a,b)\subset \Lambda_r$, it follows from Lemma~\ref{LM27} that $(a,b)\subset \Lambda_t$. The part Dirichlet forms of $(\frac{1}{2}\mathbf{D}, H^1(\mathbb{R}))$ and $(\EE,\FF)$ on $(a,b)$ are denoted by $(\frac{1}{2}\mathbf{D}_{(a,b)}, H^1_0((a,b)))$ and $(\EE_{(a,b)},\FF_{(a,b)})$. Clearly,  $(\frac{1}{2}\mathbf{D}_{(a,b)}, H^1_0((a,b)))$ is still a regular Dirichlet subspace of $(\EE_{(a,b)},\FF_{(a,b)})$ by Proposition~\ref{PRO25}. However, since $X$ stays at the starting point in $(a,b)$ forever (Cf. Lemma~\ref{LM25}~(2)), it follows that $\EE_{(a,b)}(f,f)=0$ for any $f\in C_c^\infty((a,b))\subset H^1_0((a,b))\subset \FF_{(a,b)}$, whereas
\[
	\frac{1}{2}\mathbf{D}_{(a,b)}(f,f)=\frac{1}{2}\int_a^b f'(x)^2dx.
\]
This leads to a contradiction. Thus $\Lambda_r$ has empty interior. Similarly, $\Lambda_l$ also has empty interior.

Suppose that $(a,b)\subset F=\Lambda_r\cup \Lambda_l$. We also assert that $(a,b)\subset \Lambda_t$, which leads to the same contradiction. In fact, it is enough to check that $(a,b)\cap \Lambda_{pr}=\emptyset$. Suppose that $x\in \Lambda_{pr} \cap (a,b)$. Since $\Lambda_r$ has empty interior, we have for any $n$ large enough, $(x, x+1/n)$ must contain a point in $\Lambda_{pl}$. Then we can take a sequence $x_n \downarrow x$ in $\Lambda_{pl}$. By Lemma~\ref{LM25}~(5), $x\in \Lambda_l$, which contradicts to $x\in \Lambda_{pr}$. Therefore, any point in $(a,b)$ must be a trap.

For the second assertion, fix any point $x\in F\setminus \{a_n,b_n:n\ge 1\}$. Suppose that $x\in \Lambda_{pr}$. Since $F$ is nowhere dense and $x$ is not an endpoint of some $(a_n,b_n)$, there exists a subsequence of intervals $(a_{n_k},b_{n_k})$ in \eqref{EQ2GNA} such that $a_{n_k}, b_{n_k}\downarrow x$ as $k\uparrow \infty$. Note that the left singular set $\Lambda_l$ is closed from the right and $x\in \Lambda_{pr}$. Hence for $k$ large enough, $a_{n_k}, b_{n_k}\in \Lambda_{pr}$. By Lemma~\ref{LM25}~(4), there exists a point $y>x$ such that $
	\mathbf{P}_x(\sigma_y<\infty)>0$.
Take $k$ large enough with $x<a_{n_k}<y$ and $a_{n_k}\in \Lambda_{pr}$. Particularly,
\[
	\mathbf{P}_x(\sigma_{a_{n_k}}<\infty) \geq \mathbf{P}_x(\sigma_y<\infty)>0.
\]
However, Lemma~\ref{LM26} implies $\mathbf{P}_x(\sigma_{a_{n_k}}<\infty)=0$ since $a_{n_k}\in \Lambda_{pr}$. This leads to a contradiction, and we conclude that $x\not\in \Lambda_{pr}$. The same reasoning shows $x\not\in\Lambda_{pl}$.
Hence $x\in \Lambda_t$. That completes the proof.
\end{proof}

Now we deal with $X$ on an interval $(a_n,b_n)$ of \eqref{EQ2GNA} with its endpoints. For convenience, we get rid of the subscript $n$ and write $(a_n,b_n)$ as $(a,b)$. Since $(a,b)$ is a regular interval, it follows from Lemma~\ref{LM25}~(7) that
\[
	\mathbf{P}_x(\sigma_y<\infty)\mathbf{P}_y(\sigma_x<\infty)>0,\quad \forall x,y\in (a,b).
\]
Thus $\mathbf{P}_x(\sigma_b<\infty)=0$ (resp. $\mathbf{P}_x(\sigma_a<\infty)=0$) for some $x\in (a,b)$ if and only if it holds for any $x\in (a,b)$.

Consider the right endpoint $b$. If $b=\infty$, take the part process of $X$ on $(a,\infty)$. It is an irreducible minimal diffusion process on $(a,\infty)$ (Cf. \cite[Example~3.5.7]{CF12}). Denote its scale function by $\tt$. The Brownian motion on $(a,\infty)$ ($a$ is the absorbing boundary) is its regular Dirichlet subspace. Thus from \cite[Theorem~4.1]{FHY10}, we know that $\tt(\infty)=\infty$. Particularly, $\infty$ is not approachable and $\mathbf{P}_x(X_t<\infty, \forall t)=1$ for any $x\in (a,\infty)$. Hereafter assume $b<\infty$. It has the following cases.
\begin{itemize}
\item[(1)] $b\in \Lambda_{pr}$.
By Lemma~\ref{LM26}, for any $x<b$, $\mathbf{P}_x(\sigma_b<\infty)=0$.
\item[(2)] $b\in \Lambda_t$. We claim that for $x<b$, $\mathbf{P}_x(\sigma_b<\infty)=0$. If, for some (equivalently, all) $x\in (a,b)$, $\mathbf{P}_x(\sigma_b<\infty)>0$, consider the part Dirichlet form $(\EE_{(a,\infty)},\FF_{(a,\infty)})$ of $(\EE,\FF)$ on $(a,\infty)$. Its associated minimal diffusion process is denoted by $X_{(a,\infty)}$. Note that $(\frac{1}{2}\mathbf{D}_{(a,\infty)}, H^1_0((a,\infty)))$ is its regular Dirichlet subspace. Clearly, $(a,b]$ is an invariant set of $X_{(a,\infty)}$ and $X_{(a,\infty)}^{(a,b]}$ corresponds to the Dirichlet form on $L^2((a,b])$:
\[
\begin{aligned}
	&\FF^{(a,b]}_{(a,\infty)}=\{f|_{(a,b]}: f\in \FF_{(a,\infty)}\},  \\
	&\EE^{(a,b]}_{(a,\infty)}(f|_{(a,b]},f|_{(a,b]})=\EE_{(a,\infty)}(f1_{(a,b]},f1_{(a,b]}),\quad f\in \FF_{(a,\infty)}.
\end{aligned}\]
By \cite[Theorem~3.5.8]{CF12}, we know that for any $g\in \FF^{(a,b]}_{(a,\infty)}$, $\lim_{x\uparrow b}g(x)=0$. It follows that for any $f\in \FF_{(a,\infty)}$, $\lim_{x\uparrow b}f(x)=0$. However, this contradicts to the fact that $C_c^\infty((a,\infty))\subset H^1_0((a,\infty))\subset \FF_{(a,\infty)}$.
\item[(3)] $b\in \Lambda_{pl}$. There are two cases.
\begin{itemize}
\item[(3i)] For some (equivalently, all) $x\in (a,b)$, $\mathbf{P}_x(\sigma_b<\infty)>0$. By Lemma~\ref{LM25}~(4), we also have $\mathbf{P}_b(\sigma_x<\infty)>0$ for any $x\in (a,b)$. Furthermore, $\mathbf{P}_x(X_t\leq b,\forall t)=1$ for any $x\in (a,b]$ by Lemma~\ref{LM25}~(2, 3).
\item[(3ii)] For some (equivalently, all) $x\in (a,b)$, $\mathbf{P}_x(\sigma_b<\infty)=0$.
\end{itemize}
\end{itemize}
We can also classify another endpoint $a$ as above. When $b$ (or $a$) is in the case (3i), we add $b$ (or $a$) to $(a,b)$ and attain a new interval $\langle a,b\rangle$. Clearly, $\langle a,b\rangle$ is an invariant set of $X$ in the sense that
\[
	\mathbf{P}_x(X_t\in \langle a,b\rangle, \forall t)=1,\ \forall x\in \langle a, b\rangle.
\]
Moreover, $X^{\langle a,b\rangle}$ is an irreducible diffusion process with no killing inside on $\langle a,b\rangle$ in the sense that
\[
	\mathbf{P}_x(\sigma_y<\infty)=0, \quad x,y\in \langle a,b \rangle.
\]
Then $X^{\langle a,b\rangle}$ is characterized by a scale function $\tt$ and the speed measure $m|_{\langle a, b\rangle}$. Note that $b\in \langle a,b\rangle $ if and only if $\mathbf{P}_x(\sigma_b<\infty)>0$. In other words, $b$ is approachable in finite time. From \cite[(3.5.13)]{CF12}, we can deduce that $b\in \langle a, b\rangle$ if and only if $\tt(b)<\infty$. Similarly we have $a\in \langle a,b \rangle$ if and only if $\tt(a)>-\infty$. On the other hand, the part process of $X^{\langle a,b\rangle}$ on $(a,b)$ is a minimal diffusion with the scale function $\tt$. Clearly, $(\frac{1}{2}\mathbf{D}_{(a,b)}, H^1_0((a,b)))$ is its regular Dirichlet subspace. By using \cite[Theorem~4.1]{FHY10} again, we have
\[
	dx\ll d\tt, \quad \frac{dx}{d\tt}=0 \text{ or }1,\quad d\tt\text{-a.e.}
\]
Therefore, after adjusting the value of $\tt$ up to a constant, we can conclude that $\tt\in \TT^0_\infty(\langle a, b\rangle)$. The associated Dirichlet form of $X^{\langle a,b\rangle}$ is expressed as \eqref{EQ2FNU} by \cite[Theorem~3.1]{FHY10}.

When we treat any interval $(a_n, b_n)$ in \eqref{EQ2GNA}, we obtain an invariant set $I_n:=\langle a_n, b_n\rangle$ of $X$ and $X^{\langle a_n,b_n\rangle}$ is an irreducible diffusion process on $I_n$ with a unique scale function $\tt_n\in \TT^0_\infty(\langle a_n, b_n\rangle)$ and the speed measure $m|_{\langle a_n,b_n \rangle}$. Finally any two intervals are disjoint. In fact suppose two intervals $\langle a_n, b_n\rangle, \langle a_m, b_m\rangle$ ($b_n\leq a_m$) have common point.  Then $a_m=b_n\in \langle a_n, b_n\rangle\cap  \langle a_m, b_m\rangle$. However $b_n\in \langle a_n, b_n\rangle$ implies $b_n\in \Lambda_{pl}$ and $a_m\in \langle a_m, b_m\rangle$ implies $a_m\in \Lambda_{pr}$, which contradicts the fact that $\Lambda_{pr}\cap \Lambda_{pl}=\emptyset$.

Note that any $x\in F\setminus \{a_n,b_n\}$ is a trap by Lemma~\ref{LM29}.  This implies $T_tf(x)=f(x)$ for any $f\in L^2(\mathbb{R})$ and $m$-a.e. $x\in F\setminus \{a_n,b_n\}$. Thus from Theorem~\ref{LM28} and Corollary~\ref{COR210}, we can deduce that $(\EE,\FF)$ is expressed as \eqref{EQ2FFL}.

Finally, we assert $m(F)=0$. Then the proof of the necessity of Theorem~\ref{THM23} is complete.

\begin{lemma}
$m(F)=0$.
\end{lemma}
\begin{proof}
Note that for an absolutely continuous function $f\in \FF^n$,
\begin{equation}\label{EQ2ENF}
	\EE^n(f,f)=\frac{1}{2}\int_{a_n}^{b_n}\left(\frac{df}{d\tt_n}\right)^2d\tt_n= \frac{1}{2}\int_{a_n}^{b_n}\left(\frac{df}{dx}\right)^2\left(\frac{dx}{d\tt_n}\right)^2d\tt_n=\frac{1}{2}\int_{a_n}^{b_n}\left(\frac{df}{dx}\right)^2dx.
\end{equation}
Since $C_c^\infty(\mathbb{R})\subset H^1(\mathbb{R})\subset \FF$, we have for any $f\in C_c^\infty(\mathbb{R})$,
\[
	\EE(f,f)=\frac{1}{2}\sum_{n\geq 1} \int_{a_n}^{b_n} \left(\frac{df}{dx}\right)^2dx=\frac{1}{2}\int_G\left(\frac{df}{dx}\right)^2dx.
\]
On the other hand,
\[
	\EE(f,f)=\frac{1}{2}\mathbf{D}(f,f)=\frac{1}{2}\int_\mathbb{R}\left(\frac{df}{dx}\right)^2dx.
	\]
It follows that
\[	
\int_F\left(\frac{df}{dx}\right)^2dx=0,\quad \forall f\in C_c^\infty(\mathbb{R}).
\]
This implies $m(F)=0$.
\end{proof}

\subsubsection{Proof of sufficiency}

In this part, we shall prove the sufficiency of Theorem~\ref{THM23}. Note that $(\EE,\FF)$ given by \eqref{EQ2FFL}, with invariant intervals $\{I_n:n\ge 1\}$ and scale function $\tt_n\in\TT^0_\infty(I_n)$, is a Dirichlet form on $L^2(\mathbb{R})$ by Theorem~\ref{LM28}. For convenience, an endpoint of $I_n$ which is included in $I_n$ is called a closed endpoint, and otherwise an open endpoint.
For any function $f\in C_c^\infty(\mathbb{R})$, it follows from $\tt_n\in \TT^0_\infty(I_n)$ and \eqref{EQ2ENF} that $f|_{I_n}\ll \tt_n$ and
\[
\EE^n(f|_{I_n}, f|_{I_n})=\frac{1}{2}\int_{I_n} f'(x)^2dx.
\]
Thus from $m(F)=0$, we can deduce that
\[
	\EE(f,f)=\frac{1}{2}\sum_{n\geq 1}\int_{I_n} f'(x)^2dx =\frac{1}{2}\int_\mathbb{R}f'(x)^2dx =\frac{1}{2}\mathbf{D}(f,f).
\]
This implies
\[
H^1(\mathbb{R})\subset \FF, \quad \EE(f,f)=\frac{1}{2}\mathbf{D}(f,f),\quad f\in H^1(\mathbb{R}).
\]
Finally, we need only to prove the Dirichlet form $(\EE,\FF)$ given by \eqref{EQ2FFL} is regular on $L^2(\mathbb{R})$.

\begin{lemma}\label{LM212}
The Dirichlet form $(\EE,\FF)$ given by \eqref{EQ2FFL} is regular on $L^2(\mathbb{R})$.
\end{lemma}
\begin{proof}
Clearly, $C_c^\infty(\mathbb{R})\subset \FF\cap C_c(\mathbb{R})$. So $\FF\cap C_c(\mathbb{R})$ is dense in $C_c(\mathbb{R})$ with the uniform norm. It suffices to prove $\FF\cap C_c(\mathbb{R})$ is dense in $\FF$ with the norm $\|\cdot\|_{\EE_1}$.

We first note that $(\EE^n,\FF^n)$ is regular on $L^2(I_n)$. Set $\mathcal{C}_n:=\FF^n\cap C_c(I_n)$. Define the following class
\[
	\mathcal{C}:= \{f\in \FF: f^n\in \mathcal{C}_n\},
\]
where $f^n:=f|_{I_n}$. Then $\mathcal{C}$ is dense in $\FF$ with the norm $\|\cdot\|_{\EE_1}$. In fact, fix $f\in \FF$ and $\epsilon>0$. For each $n$, take a function $g_n\in \mathcal{C}_n$ such that $\|f^n-g_n\|^2_{\EE^n_1}<\epsilon/2^n$. Let $g$ be the function: $g|_{I_n}=g_n$, $g=0$ outside $\cup_{n\geq 1}I_n$. Clearly $g\in \FF$ and hence $g\in \mathcal{C}$. Furthermore,
\[
	\EE_1(f-g,f-g)=\sum_{n\geq 1}\EE^n_1(f^n-g_n,f^n-g_n) <\epsilon.
\]
Therefore, we need only to prove $\FF\cap C_c(\mathbb{R})$ is dense in $\mathcal{C}$ with the norm $\|\cdot\|_{\EE_1}$.

Fix a function $f\in \mathcal{C}$ and a constant $\epsilon>0$. There exists an integer $n$ large enough such that $\sum_{k>n}\EE_1^k(f^k,f^k)<\epsilon $. Let $g:=f\cdot 1_{\cup_{k=1}^n\langle a_k, b_k\rangle}$. Then $g\in \mathcal{C}\subset \FF$ and \[
	\EE_1(f-g,f-g)=\sum_{k>n}\EE_1^k(f^k,f^k)<\epsilon.
\]
We need now to find a function in $\FF\cap C_c(\mathbb{R})$ which is $\EE_1$-close enough to $g$.

Note that $g$ is continuous on $\langle a_k, b_k\rangle$. The discontinuous points of $g$ are those closed endpoints of $\{I_k: 1\leq k\leq n\}$. Particularly, the discontinuous points of $g$ are finite. Take such a discontinuous point $c$ of $g$. Without loss of generality, assume that $c$ is the right endpoint of some interval $\langle a_k, b_k\rangle$ with $b_k\in \langle a_k, b_k\rangle$.  Set $h:=g(c)$. There are two different situations.
\begin{itemize}
\item[(1)] For any $\beta>0$, there exists an open endpoint of $\{I_n:n\ge 1\}$ in $[c,c+\beta)$, i.e., $c$ is a limit point of open endpoints of $\{I_n:n\ge 1\}$.
\begin{itemize}
\item[(1a)] Let us start from a simple case where
$c$ is an open endpoint of some $I_{k'}$, which is essentially the same as the example given in introduction.

In this case $c=a_{k'}\notin \langle a_{k'},b_{k'}\rangle$ and $\tt_{k'}(c)=-\infty$. Since $g^{k'}:=g|_{( a_{k'},b_{k'}\rangle}\in C_c((a_{k'},b_{k'}\rangle)$, we can take $d\in (a_{k'}, b_{k'})$ such that $g=0$ on $(c, d]$.  We shall construct a continuous function $\varphi=\varphi_c^\epsilon$ on $[c, d]$ (set $\varphi(x)=0$ if $x\notin [c,d]$) such that
\begin{equation}\label{EQ2VCC}
\varphi(c)=h, \quad \varphi(d)=0,\quad \varphi \in \FF,\quad \EE_1(\varphi, \varphi)<\frac{\epsilon}{2n}.
\end{equation}
Obviously $g+\varphi$ will be continuous at $c$ and its $\EE_1$-distance to $g$ is small.

Take a constant $\delta>c$ such that $h^2\cdot (\delta-c)<\epsilon/4n$. Let $\tilde{\delta}:=\tt_{k'}(\delta)>-\infty$. Take another constant $\tilde{\delta}'<\tilde{\delta}-(8h^2n)/\epsilon$ and let $\delta':=\tt_{k'}^{-1}(\tilde{\delta}')$. Clearly there exists a $C^1$-function $\phi$ on $[\tilde{\delta}', \tilde{\delta}]$ such that
\[
	0\leq \phi \leq h,\quad \phi(\tilde{\delta}')=h,\quad   \phi(\tilde{\delta})=0,\quad |\phi'|\leq \frac{2h}{\tilde{\delta}-\tilde{\delta}'}.
\]
Define $$\varphi(x):=\begin{cases}h, & x\in [c, \delta']\\
\phi(\tt_{k'}(x)),& x\in [\delta', \delta]\\
0,&x>\delta. \end{cases}$$
Clearly, $\varphi$ is continuous on $[c,d]$ and $\varphi\in \FF$. Furthermore,
\[
\begin{aligned}
	\EE_1(\varphi,\varphi)&=\frac{1}{2}\int_{\delta'}^{\delta}\left( \frac{d\varphi}{d\tt_{k'}}\right)^2d\tt_{k'}+\int_c^\delta \varphi_c(x)^2dx\\
	&=\frac{1}{2}\int_{\tilde{\delta}'}^{\tilde{\delta}}\left(\phi'(x)\right)^2dx+\int_c^\delta \varphi(x)^2dx \\
	&\leq \frac{1}{2}\frac{4h^2}{\tilde{\delta}-\tilde{\delta}'} + h^2\cdot(\delta-c) \\
	&<\frac{\epsilon}{2n}.
\end{aligned}\]
Therefore, $\varphi$ satisfies \eqref{EQ2VCC}. See Figure~\ref{UC}.
\begin{figure}
\centering
\includegraphics[scale=0.55]{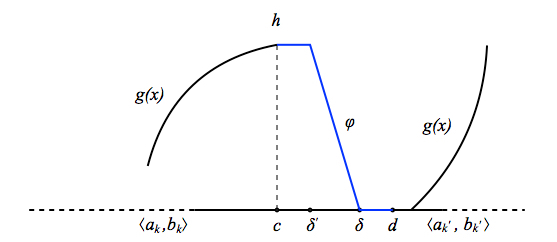}
\caption{Compensate function $\varphi$}\label{UC}
\end{figure}

\item[(1b)] $c$ is a limit point of open endpoints of $\{I_n:n\ge 1\}$.

We see that the main reason that $\varphi$ above can be constructed is that there is a non-closed interval $I_k$ close to $c$,
because in this case it follows from $g|_{I_k}\in C_c(I_k)$ that $g$ vanishes on an interval contained in $I_k$.
More precisely we can take an non-closed interval $\langle a_q, b_q\rangle $, where $b_q$ is an open endpoint, such that
\[
b_q-c<\epsilon/(4nh^2)
\]
and $g=0$ on $(c,b_q]$. Similarly to the first case, we can also construct a continuous function $\varphi=\varphi_c^\epsilon$ on $[c,b_q]$ ($\varphi:=0$ outside $[c,b_q]$) such that
\[
	\varphi(c)=h, \quad  \varphi(b_q)=0,\quad \varphi \in \FF,\quad \EE_1(\varphi, \varphi)<\frac{\epsilon}{2n}.
\]
\end{itemize}

\item[(2)] For some $\beta>0$, the endpoints of $\{I_n:n\ge1\}$ located between $c$ and $c+\beta$ are all closed.

Let $\beta$ be small enough so that $g=0$ on $(c, c+\beta]$ and $$\beta<\frac{\epsilon}{2nh^2}.$$
Denote the intervals of $\{\langle a_k,b_k\rangle: k>n\}$ in $(c, c+\beta)$ by $\{[a_{q_j},b_{q_j}]: j\geq 1\}$. Note that they are disjoint and dense in $[c,c+\beta]$. Hence $\bigcup_j [a_{q_j},b_{q_j}]$ is a Cantor-type set. We may construct a Cantor-type function $\phi$ on $[c,c+\beta]$ (for its existence, see Remark~\ref{RM212}), such that $\phi$ is decreasing and continuous, $\phi(c)=1, \phi(c+\beta)=0$ and $\phi$ is a constant on each interval $[a_{q_j}, b_{q_j}]$. Define $\varphi=\varphi_c^\epsilon(x):=h\cdot \phi(x)$ for $x\in [c, c+\beta]$ and vanishes elsewhere. Clearly, $\varphi\in \FF$ and $$\EE_1(\varphi,\varphi)=\int\varphi(x)^2dx<\epsilon/(2n).$$ Thus $\varphi$ satisfies \eqref{EQ2VCC} if $d$ is replaced by $c+\beta$.
\end{itemize}

From the above discussions, we can always construct a compensate function $\varphi_c^\epsilon$ which depends on discontinuous point $c$ of $g$ and $\epsilon$. The construction above may guarantee that for any $c\neq c'$, $\varphi_c^\epsilon$ and $\varphi_{c'}^\epsilon$ have disjoint supports.  Define
\[
	f_\epsilon:=g+\sum_{c} \varphi_c^\epsilon,
\]
where $c$ in the sum takes all possible discontinuous points of $g$. The number of the terms in this sum is less than $2n$. One may easily check that $f_\epsilon \in \FF\cap C_c(\mathbb{R})$. Therefore
\[
	\EE_1(f_\epsilon-g,f_\epsilon-g)=\EE_1\left(\sum_c\varphi_c^\epsilon, \sum_c\varphi_c^\epsilon\right)
=\sum_c\EE_1(\varphi_c^\epsilon,\varphi_c^\epsilon) \leq 2n\cdot \frac{\epsilon}{2n}=\epsilon.
\]
That completes the proof.
\end{proof}
\begin{remark}\label{RM212}
In this remark, we give a Cantor-type function $\phi$ on $[c, c+\beta]$ which is used in the proof of Lemma~\ref{LM212}, though it may be found in some textbook. Without loss of generality, assume that $[c, c+\beta]=[0,1]$, $\{I_n=[a_n,b_n]:n\geq 1\}$ are disjoint closed intervals in $(0,1)$ and $m\left([0,1]\setminus \cup_{n\geq 1}I_n \right)=0$. The continuous function $\phi$ on $[0,1]$ is desired to satisfy $\phi(0)=1, \phi(1)=0$ and $\phi$ is a constant on each $I_n$.

Rearrange the positive integers as the following way:
\[
\begin{aligned}
	&K_1:=\{k_1:=1\},  \\
&K_2:=\{k_{2,1}:=\min\{n: a_n<a_{k_1}\}, k_{2,2}:=\min\{n: a_n>a_{k_1}\}\},
\end{aligned}
\]
Assume that the sets $K_1,\cdots,K_{m-1}$ have been defined. Then $$(0,1) \setminus \bigcup\left\{I_n: n\in \bigcup_{j=1}^{m-1}K_j\right\}$$ is separated into $2^{m-1}$
disjoint and connected open intervals. Let $k_{m,i}$ be the smallest integer $n$ of $I_n$ in the $i$-th interval from left to right for $1\le i\le 2^{m-1}$.
Define inductively $K_m=\{k_{m,i}:1\le i\le 2^{m-1}\}$. Clearly
\[
	a_{k_{m,1}}<a_{k_{m,2}}<\cdots  <a_{k_{m, 2^{m-1}}},
\]
and $\mathbb{N}=\bigcup_{m\geq 1} K_m$.
Define the function $\phi$ as follows: $\phi(0):=1$, $\phi(1):=0$ and for any $m\geq 1, 1\leq j\leq 2^{m-1}$,
\[
	\phi(x):=\frac{2^{m-1}-j}{2^{m-1}},\quad \forall x\in I_{k_{m,j}}.
\]
One may prove that $\phi$ can be extended to a decreasing and continuous function on $[0,1]$ similar to the standard Cantor function on $[0,1]$.
\end{remark}

\subsection{More examples of extension}\label{SEC23}

In this section, we give several examples for the regular Dirichlet extensions of one-dimensional Brownian motion. Recall that the regular Dirichlet extension $(\EE,\FF)$ is characterized by $\{I_n:n\geq 1\}$ and $\{\tt_n\in \mathbf{T}^0_\infty(I_n):n\geq 1\}$ in Theorem~\ref{THM23}. It is evident
that the extension $(\EE,\FF)$ is same as Brownian motion if and only if only invariant interval is $\mathbb{R}$ and the scale function $\tt(x)=x$.

\begin{example}\label{EXA215}
Let $I_1=\mathbb{R}$ and $\tt_1(x)=x+c(x)$, where $c(x)$ is the standard Cantor function on $[0,1]$ and we set $c(x):=0$ for $x\leq 0$ and $c(x):=1$ for $x\geq 1$.  Then the corresponding regular Dirichlet extension is irreducible.
\end{example}

\begin{example}
Let $I_1:=(-\infty, 0)$ and $I_2:=(0,\infty)$. Take $\tt_1\in \TT^0_\infty(I_1)$ and $\tt_2\in \TT^0_\infty(I_2)$. Then $I_1$ and $I_2$ are two invariant sets of $X$. The single point set $\{0\}$ is an $m$-polar set relative to $X$. Formally, we may assume $0$ is a trap of $X$, i.e. $\mathbf{P}_0(X_t=0,\forall t)=1$.
\end{example}

\begin{example}
Let $I_1:=(-\infty,-1], I_2:=[1,\infty)$, $I_{2k+1}:=(-\frac{1}{k},-\frac{1}{k+1}]$ and $I_{2k+2}:=[\frac{1}{k+1}, \frac{1}{k} )$ for any $k \geq 1$. Take $\tt_n\in \TT_\infty^0(I_n)$ for each $n$. Then $\{0\}$ is an $m$-polar set relative to $X$ and any other single point set is not $m$-polar.
\end{example}

\begin{example}
Let $I_1:=(-\infty, 0]$, $I_2=(1,\infty)$, and $I_n:=(\frac{1}{n-1}, \frac{1}{n-2})$ for any $n\geq 3$. Take $\tt_n\in \TT_\infty^0(I_n)$ for each $n$. Then $\{\frac{1}{k}:k\geq 1\}$ is an $m$-polar set relative to $X$.
\end{example}

\begin{example}\label{EXA218}
Let $K$ be the standard Cantor set in $[0,1]$. Set $U:=K^c$ and write $U$ as a union of disjoint open intervals:
\[
U=\bigcup_{n\geq 1}(a_n,b_n),
\]
where $(a_1,b_1)=(-\infty,0), (a_2,b_2)=(1,\infty)$. Let $I_1:=(-\infty, 0], I_2:=[1,\infty)$, $I_n:=[a_{n},b_{n}]$ for any $n\geq 3$. For each $n$, let $\tt_n(x)=x$ on $I_n$. Then the associated diffusion process $X$ of this regular Dirichlet extension is a reflected Brownian motion on each interval $I_n$. Moreover,
\[
	K\setminus \left \{a_n,b_n: n\geq 1 \right\}
\]
is $m$-polar.
\end{example}

\section{{Structures of regular Dirichlet extensions: orthogonal complements and darning processes}}\label{SEC3}

In \cite{LY14}, the structures of regular Dirichlet subspaces for one-dimensional Brownian motion were investigated by using the trace method and a darning transform.
As we have seen, `trace method' could efficiently trace the different behavior of regular subspace from Brownian motion.
In this section, we shall apply the same approach to investigate the behavior of regular Dirichlet extensions of one-dimensional Brownian motion. The Dirichlet form $(\EE,\FF)$ always stands for a proper regular Dirichlet extension of $(\frac{1}{2}\mathbf{D}, H^1(\mathbb{R}))$ on $L^2(\mathbb{R})$, which is characterized by Theorem~\ref{THM23}. If not otherwise stated, $m$ denotes the Lebesgue measure on $\mathbb{R}$ in this section.

\subsection{Orthogonal complement of Brownian motion}\label{SEC31}

Let us characterize the orthogonal complement of one-dimensional Brownian motion in extension space. We need first to formulate
extended Dirichlet space.
The extended Dirichlet space of $(\frac{1}{2}\mathbf{D}, H^1(\mathbb{R}))$ is
\[
	H^1_\mathrm{e}(\mathbb{R}):=\left\{f: f \text{ is absolutely continuous and }f'\in L^2(\mathbb{R})\right\}.
\]
The extended Dirichlet space of $(\EE^n,\FF^n)$ given by \eqref{EQ2FNU} is expressed as (Cf. \cite[Theorem~2.2.11]{CF12})
\begin{equation}\label{EQ3FNE}
\FF^n_\mathrm{e}=\left\{f\text{ on }I_n: f\ll \tt_n, \int_{I_n}\left(\frac{df}{d\tt_n}\right)^2d\tt_n<\infty \right\}.
\end{equation}
We formulate the extended Dirichlet space of the regular Dirichlet extension \eqref{EQ2FFL} in the following theorem.

\begin{theorem}\label{LM31}
The extended Dirichlet space of $(\EE,\FF)$ given by \eqref{EQ2FFL} is
\begin{equation}\label{EQ3FEF}
\FF_\mathrm{e}=\left\{f: |f|<\infty\; m \text{-a.e. on}\; \mathbb{R}, f|_{I_n}\in \FF^n_\mathrm{e}, \sum_{n\geq 1}\EE^n(f|_{I_n}, f|_{I_n})<\infty  \right\}.
\end{equation}
\end{theorem}
\begin{proof}
Take an arbitrary $f\in \FF_\mathrm{e}$. Clearly, $|f|<\infty$ $m$-a.e. on $\mathbb{R}$. By the definition of extended Dirichlet space (Cf. \cite[Definition~1.1.4]{CF12}), there exists an $\EE$-Cauchy sequence $\{f_l\}\subset \FF$ such that $\lim_{l\rightarrow \infty}f_l=f$ $m$-a.e. on $\mathbb{R}$. Particularly, for each $n$, $\{f_l|_{I_n}\}\subset \FF^n$ is $\EE^n$-Cauchy and $\lim_{l\rightarrow \infty}f_l|_{I_n}=f|_{I_n}$ $m|_{I_n}$-a.e. on $I_n$. This implies $f|_{I_n}\in \FF^n_\mathrm{e}$ and
\[
	\EE^n(f|_{I_n},f|_{I_n})=\lim_{l\rightarrow \infty} \EE^n(f_l|_{I_n},f_l|_{I_n}).
\]
On the other hand, since $\{f_l\}$ is $\EE$-Cauchy, we may take an integer $M$ large enough such that for any $l>M$,
\[
	\EE(f_M-f_l, f_M-f_l)<1.
\]
Then we have
\[
\begin{aligned}
\sum_{n\geq 1}\EE^n(f|_{I_n},f|_{I_n})&=\sum_{n\geq 1}\lim_{l\rightarrow \infty}\EE^n(f_l|_{I_n},f_l|_{I_n}) \\
&\leq \liminf_{l\rightarrow \infty} \sum_{n\geq 1} \EE^n(f_l|_{I_n}, f_l|_{I_n}) \\
&=\liminf_{l\rightarrow \infty} \EE(f_l,f_l) \\
&\leq \liminf_{l\rightarrow \infty} 2\left( \EE(f_l-f_M,f_l-f_M)+\EE(f_M,f_M) \right) \\
&\leq 2\left(1+\EE(f_M,f_M)\right) \\
&<\infty.
\end{aligned}
\]
This indicates $f$ is in the right side of \eqref{EQ3FEF}.

 On the contrary, let $f$ be a function in the right side of \eqref{EQ3FEF}. Since $f|_{I_n}\in \FF^n_\mathrm{e}$, we may take an $\EE^n$-Cauchy sequence $\{g^n_l: l\geq 1\}\subset \FF^n$ such that $g^n_l\rightarrow f|_{I_n}$ $m$-a.e. as $l\rightarrow \infty$. Particularly,
\[
	\lim_{l\rightarrow \infty}\EE^n(g^n_l-f|_{I_n}, g^n_l-f|_{I_n})=0.
\]
Thus for each positive integer $k$, there are two integers $l_k^n$ and $N_k$ such that
\[
\begin{aligned}
	&\EE^n(g^n_{l^n_k}-f|_{I_n}, g^n_{l^n_k}-f|_{I_n})<\frac{1}{k}\cdot \frac{1}{2^n}, \\
	&\sum_{n>N_k}\EE^n(f|_{I_n},f|_{I_n})<\frac{1}{k}.
\end{aligned}\]
Without loss of generality, we may assume $l^n_k, N_k\uparrow \infty$ as $k\rightarrow \infty$.
Define a function $h_k$ $m$-a.e. on $\mathbb{R}$: $h_k|_{I_n}:= g^n_{l^n_k}$ for any $1\leq n\leq N_k$ and $h_k:=0$ elsewhere. Clearly, $h_k\in L^2(\mathbb{R})$ and $h_k$ converges to $f$ $m$-a.e. as $k\rightarrow \infty$. Note that $g^n_{l^n_k}\in \FF^n$ and
\[
\begin{aligned}
	\sum_{n= 1}^{N_k}\EE^n(g^n_{l^n_k}, g^n_{l^n_k})& \leq 2\sum_{n= 1}^{N_k} \left(\EE^n(g^n_{l^n_k}-f|_{I_n}, g^n_{l^n_k}-f|_{I_n})+\EE^n(f|_{I_n},f|_{I_n}) \right)  \\
	&< 2\sum_{n\geq 1} \frac{1}{k}\cdot \frac{1}{2^n} + 2\sum_{n\geq 1}\EE^n(f|_{I_n},f|_{I_n})  \\	
	&<\infty.
\end{aligned}
\]
This implies $h_k\in \FF$.  Finally, we show that $\{h_k: k\geq 1\}$ is $\EE$-Cauchy in $\FF$. In fact, for any $\epsilon>0$, take an integer $K$ satisfying $8/K<\epsilon$. Then for any $k, k'>K$, we have
\[
\begin{aligned}
\EE(h_k-h_{k'},h_k-h_{k'})&\leq 2\left(\EE(h_k-f,h_k-f)+\EE(h_{k'}-f,h_{k'}-f) \right) \\
&\leq 2\left(\sum_{n=1}^{N_k}\EE^n(g^n_{l^n_k}-f|_{I_n},g^n_{l^n_k}-f|_{I_n}) +\sum_{n>N_k}\EE^n(f|_{I_n},f|_{I_n}) \right) \\
&\qquad +2\left(\sum_{n=1}^{N_{k'}}\EE^n(g^n_{l^n_{k'}}-f|_{I_n},g^n_{l^n_{k'}}-f|_{I_n}) +\sum_{n>N_{k'}}\EE^n(f|_{I_n},f|_{I_n}) \right) \\
&\leq 4\sum_{n\geq 1}\frac{1}{K}\cdot \frac{1}{2^n}+\frac{4}{K}\\
&<\epsilon.
\end{aligned}\]
That completes the proof.
\end{proof}

The purpose of the next part is to formulate the orthogonal complement of $H^1_\mathrm{e}(\mathbb{R})$ in $\FF_\mathrm{e}$ in $(\EE,\FF)$.
For each $n\geq 1$, denote
\begin{equation}\label{EQ3UNX}
\begin{aligned}
	&U_n:=\left\{x\in I_n: \frac{dx}{d\tt_n}(x)=1\right\}, \\
	&W_n:=I_n\setminus U_n.
\end{aligned}\end{equation}
Then $U_n, W_n$ are defined in the sense of $d\tt_n$-a.e. Since the natural scale is strictly increasing and continuous, it follows that $U_n$ is measurable dense in $I_n$ in the sense that
\[
	d\tt_n(U_n\cap (c,d))>0,\quad \forall (c,d)\subset I_n.
\]
Particularly,
\[
	m|_{I_n}=1_{U_n}d\tt_n,\quad d\tt_n(U_n)=m(U_n)=|b_n-a_n|, \quad  m(W_n)=0.
\]
 Define
\begin{equation}\label{EQ3GFFE}
\mathcal{G}:=\{f\in \FF_\mathrm{e}: \EE(f,g)=0, \forall g\in H^1_\mathrm{e}(\mathbb{R})\}.
\end{equation}
We write $\FF_\mathrm{e}=H^1_\mathrm{e}(\mathbb{R})\oplus \mathcal{G}$ or $\mathcal{G}=\FF_\mathrm{e}\ominus H^1_\mathrm{e}(\mathbb{R})$.

\begin{theorem}\label{THM32}
The orthogonal complement $\mathcal{G}$ of $H^1_\mathrm{e}(\mathbb{R})$ in $\FF_\mathrm{e}$ is expressed as
\begin{equation}\label{EQ3GFF}
	\mathcal{G}=\left\{f\in \FF_\mathrm{e}: \frac{df|_{I_n}}{d\tt_n}=0, d\tt_n\text{-a.e. on }U_n\text{ for any }n\geq 1\right\}.
\end{equation}
Furthermore, any $f\in \FF_\mathrm{e}$ can be decomposed into
\begin{equation}\label{EQ3UUU}
	f=f_1+f_2,\quad f_1\in H^1_\mathrm{e}(\mathbb{R}), f_2\in \mathcal{G}.
\end{equation}
This decomposition is unique up to a constant. In other words, $H^1_\mathrm{e}(\mathbb{R})\cap \mathcal{G}$ only contains the constant functions.
\end{theorem}

\begin{proof}
We first prove the expression \eqref{EQ3GFF} of $\mathcal{G}$. Fix a function $f$ in the right side of \eqref{EQ3GFF} and take another function $g$ in $H^1_\mathrm{e}(\mathbb{R})$. We have
\[
	\EE(f,g)=\sum_{n\geq 1}\EE^n(f|_{I_n}, g|_{I_n})=\frac{1}{2}\sum_{n\geq 1} \int_{I_n}\frac{df|_{I_n}}{d\tt_n}(x) g'(x)1_{U_n}(x)d\tt_n(x)=0.
\]
It follows that $f\in \mathcal{G}$. On the contrary, take an arbitrary function $f\in \mathcal{G}$. Note that $I_n=\langle a_n, b_n\rangle$. Any function in $C_c^\infty((a_n,b_n))$ is treated as a function on $\mathbb{R}$ and clearly $C_c^\infty((a_n,b_n)) \subset H^1_\mathrm{e}(\mathbb{R})$. From \eqref{EQ3GFFE}, we have
\[
	\EE(f,g)=0,\quad \forall g\in C_c^\infty((a_n, b_n)).
\]
It follows that
\[
\int_{a_n}^{b_n}\left(\frac{df|_{I_n}}{d\tt_n}(x)1_{U_n}(x) \right)g'(x)dx=0,\quad \forall g\in C_c^\infty((a_n, b_n)).
\]
This implies that $df|_{I_n}/d\tt_n\cdot 1_{U_n}$ is a constant a.e. on $(a_n,b_n)$, or equivalently $df|_{I_n}/d\tt_n$ is constant $d\tt_n$-a.e. on $U_n$. Denote this constant by $c_n$. Take two integers $m,n$ so that $a_n<a_m$. Define a function $h$ on $\mathbb{R}$:
$$h(x):= (x-a_n)1_{I_n}(x)+|I_n|\cdot 1_{[b_n,a_m]}(x)+\left[|I_n|-\frac{|I_n|}{|I_m|}\cdot (x-a_m)\right]\cdot
1_{[a_m, b_m]}(x).$$ (see Figure~\ref{FIG2}). Clearly, $h\in H^1_\mathrm{e}(\mathbb{R})$. Hence we have
\[
\begin{aligned}
0&=\EE(f,h) \\
&=\frac{1}{2}\left(\int_{I_n}\frac{df|_{I_n}}{d\tt_n}(x)h'(x)1_{U_n}(x)d\tt_n(x)+ \int_{I_m}\frac{df|_{I_m}}{d\tt_m}(x)h'(x)1_{U_m}(x)d\tt_m(x) \right) \\
&= \frac{1}{2}\left(c_n\int_{I_n}h'(x)dx+c_m\int_{I_m}h'(x)dx \right) \\
&=\frac{b_n-a_n}{2}\cdot (c_n-c_m).
\end{aligned}\]
It follows that $c_n=c$ for some constant $c$ and any $n\geq 1$. On the other hand, the fact $f\in \FF_\mathrm{e}$ implies
\[
	\EE(f,f)<\infty.
\]
However,
\[
	\EE(f,f)=\sum_{n\geq 1}\EE^n(f|_{I_n},f|_{I_n})\geq \frac{c^2}{2}\sum_{n\geq 1} \int_{I_n} 1_{U_n}(x)d\tt_n(x)=\frac{c^2}{2}m(\mathbb{R}).
\]
Therefore, $c=0$ and $f$ is in the right side of \eqref{EQ3GFF}.

\begin{figure}
\includegraphics[scale=0.55]{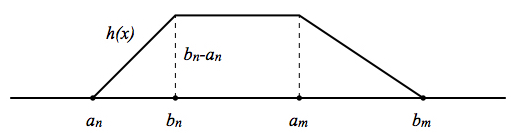}
\caption{The function $h$}\label{FIG2}
\end{figure}

Next, we prove the decomposition \eqref{EQ3UUU}. Fix a function $f\in \FF_\mathrm{e}$. We decompose $f|_{I_n}$ for any $n\geq 1$ as
\[
	f|_{I_n}=g^n_1+g^n_2,
\]
where $g^n_1\in H^1_\mathrm{e}(\mathbb{R})$ is supported on $I_n$ and $g^n_1(a_n)=0$ (resp. $g^n_1(b_n)=0$) if $a_n>-\infty$ (resp. $b_n<\infty$), $dg^n_2/d\tt_n$ is a constant $d\tt_n$-a.e. on $U_n$. In fact, let $e_n$ be a fixed point in $(a_n,b_n)$. If $(a_n,b_n)=(-\infty, \infty)$, set
\[
\begin{aligned}
	&g^n_1(x):=\int_{e_n}^x \frac{df}{d\tt_n}(x)1_{U_n}(x)d\tt_n(x), \\
	&g^n_2(x):=f|_{I_n}(x)-g^n_1(x)=\int_{e_n}^x \frac{df}{d\tt_n}(x)1_{W_n}(x)d\tt_n(x)+f(e_n)
\end{aligned}\]
for any $x\in I_n$. If $a_n$ is finite but $b_n=\infty$ (the case $a_n=-\infty$ and $b_n<\infty$ is similar), set
\[
\begin{aligned}
&g^n_1(x):= \int_{e_n}^x \frac{df}{d\tt_n}(x)1_{U_n}(x)d\tt_n(x)+C, \\
&g^n_2(x):=f|_{I_n}(x)-g^n_1(x)= \int_{e_n}^x \frac{df}{d\tt_n}(x)1_{W_n}(x)d\tt_n(x)+f(e_n)-C
\end{aligned}
\]
for any $x\in I_n$, where
\[
	C:=\int_{a_n}^{e_n} \frac{df}{d\tt_n}(x)1_{U_n}(x)d\tt_n(x).
\]
Note that
\[
|C|\leq \left(\int_{a_n}^{e_n}\left(\frac{df}{d\tt_n}(x)\right)^2d\tt_n(x)\right)^{1/2}\cdot |e_n-a_n|^{1/2}<\infty
\]	
and $g^n_1(a_n)=0$. When $I_n$ is not finite, let $C^n_1:=0$. If $a_n$ and $b_n$ are both finite, then
\[
	M:= \int_{I_n}\frac{df}{d\tt_n}(x)1_{U_n}(x)d\tt_n(x)
\]
is finite. Set $C^n_1:=M/(b_n-a_n)$ and
\[
	C^n_2:= \int_{a_n}^{e_n}\left( \frac{df}{d\tt_n}(x)-C^n_1\right)1_{U_n}(x)d\tt_n(x).
\]
Clearly, $C^n_2$ is also finite. Define
\[
\begin{aligned}
	&g^n_1(x):=\int_{e_n}^x \left( \frac{df}{d\tt_n}(y)-C^n_1\right)1_{U_n}(y)d\tt_n(y)+C^n_2, \\
	&g^n_2(x):=f|_{I_n}(x)-g^n_1(x).
\end{aligned}\]
It is easily seen that $\lim_{x\downarrow a_n}g^n_1(x)=\lim_{x\uparrow b_n}g^n_1(x)=0$. For all three cases above, we may easily deduce that $g^n_1\in H^1_\mathrm{e}(\mathbb{R})$  and $dg^n_2/d\tt_n=C^n_1$ $d\tt_n$-a.e. on $U_n$. Then we define a function $f_0$ on $\mathbb{R}$ as follows: $f_0(x):=g^n_1(x)$ for any $x\in I_n$ and $n\geq 1$ and $f_0(x):=0$ elsewhere. It follows that $f_0\in H^1_\mathrm{e}(\mathbb{R})$. Next define $h|_{I_n}:=C^n_1$ for any $n\geq 1$. Since
\[
\int_\mathbb{R}h^2(x)dx=\sum_{n\geq 1} \left(\frac{M}{b_n-a_n}\right)^2\cdot (b_n-a_n)\leq \sum_{n\geq 1}\int_{I_n}\left(\frac{df}{d\tt_n}\right)^2d\tt_n<\infty,
\]
we conclude that $h\in L^2(\mathbb{R})$ is locally integrable. Let
\[
\begin{aligned}
	&f_1(x):=f_0(x)+\int_0^xh(y)dy,\quad x\in \mathbb{R}, \\
	&f_2:=f-f_1.
\end{aligned}\]
Then we have $f_1\in H^1_\mathrm{e}(\mathbb{R})$  and thus $f_2\in \FF_\mathrm{e}$. On the other hand,
\[
	\frac{df_2|_{I_n}}{d\tt_n}=\frac{df|_{I_n}}{d\tt_n}-\frac{df_0|_{I_n}}{d\tt_n}-C_1^n=\frac{dg^n_2}{d\tt_n}-C^n_1=0,
\]
$d\tt_n$-a.e. on $U_n$. Hence $f_2\in \mathcal{G}$ by \eqref{EQ3GFF}.

Finally, assume that $f\in H^1_\mathrm{e}(\mathbb{R})\cap \mathcal{G}$. It follows from \eqref{EQ3GFFE} that
\[
	0=\EE(f,f)=\sum_{n\geq 1}\EE^n(f|_{I_n},f|_{I_n}).
\]
This implies that $\EE^n(f|_{I_n},f|_{I_n})=0$. Since $f|_{I_n}\in \FF^n_\mathrm{e}$ and $(\EE^n,\FF^n)$ is irreducible,  we conclude from \cite[Theorem~5.2.16]{CF12} that $f|_{I_n}$ is a constant on $I_n$. Then we have $f'=0$ on $\cup_{n\geq 1}I_n$ and hence $m$-a.e. on $\mathbb{R}$. Therefore, $f$ is a constant on $\mathbb{R}$. That completes the proof.
\end{proof}

\begin{remark} One may feel that the decomposition \eqref{EQ3UUU} is obvious by applying the orthogonal decomposition theorem in Hilbert space. However, though the terminology `orthogonal complement' is used here, we should notice that $(\EE,\FF_e)$ is not a Hilbert space. For $f\in \FF_e$ with $\EE(f,f)=0$, $f$ may not be necessarily a constant. Hence the decomposition \eqref{EQ3UUU} can not be deduced simply from the orthogonal decomposition of Hilbert space.
\end{remark}

\begin{example}
In this example, let us consider the regular Dirichlet extension $(\EE,\FF)$ stated in Example~\ref{EXA218}. Note that the associated diffusion process $X$ is a reflected Brownian motion on each interval $I_n$ and $U_n=I_n$. Then the extended Dirichlet space of $(\EE,\FF)$ is expressed as
\[
	\FF_\mathrm{e}=\left\{ f: f \text{ is absolutely continuous on each interval }I_n\text{ and } \sum_{n\geq 1}\int_{I_n}f'(x)^2dx<\infty\right\}.
\]
Moreover,
\[
	\mathcal{G}=\FF_\mathrm{e}\ominus H^1_\mathrm{e}(\mathbb{R})= \left\{f: f \text{ is a constant on each interval }I_n \right\}.
\]
The orthogonal complement $\mathcal{G}$ contains continuous  functions as well as discontinuous functions. For example, the Cantor-type function introduced in Remark~\ref{RM212} belongs to $\mathcal{G}$.
\end{example}

\subsection{Darning processes}\label{SEC42}

{
Recall the definitions of $U_n$ and $W_n$ in \eqref{EQ3UNX}. From now on, we impose the following assumptions on $U_n$:
\begin{description}
\item[(H1)] $U_n$ has (and is taken as) a $d\tt_n$-a.e. open version;
\item[(H2)] for any $x\in W_n\cap (a_n,b_n)$ and $\epsilon>0$, $d\tt_n\left( (x-\epsilon, x+\epsilon)\cap W_n\right)>0$.
\end{description}
The first assumption is not always right and the second one is not essential as we remarked in \cite[\S1]{LY14}. In fact, if \textbf{(H1)} is satisfied, we can always find an open $d\tt_n$-version of $U_n$ that satisfies \textbf{(H2)}, see also \cite[\S1]{LY14}.
Write
\begin{equation}\label{EQ3UNM}
U_n=\bigcup_{m\geq 1}(a_m^n, b_m^n)
\end{equation}
as a union of disjoint open intervals and set
\[
\begin{aligned}
	&U:=\bigcup_{n\geq 1}U_n= \bigcup_{n\geq 1}\bigcup_{m\geq 1}(a^n_m,b^n_m),\\
	&K:=U^c.
\end{aligned}\]

\begin{remark}\label{RM34}
We need to give some explanation for the structures of $U_n$ and $K$.  Now we only consider the right endpoint $b_n$ of $I_n$ (the case of the left endpoint $a_n$ is similar).  We first note that  $b_n\notin U_n$ if  $b_n\in I_n$, since $U_n$ is assumed to be open in $\mathbb{R}$. If $b_n\in I_n$, then it may happen that $b_n=b^n_m$ for some integer $m$ in \eqref{EQ3UNM}. For instance, in Example~\ref{EXA218}, we have $U_n=\overset{\circ}{I}_n=(a_n,b_n)$. If $b_n\notin I_n$ and $b_n<\infty$, then $W_n$ is not trivial and $d\tt_n(W_n)=\infty$. This follows from $\tt_n(b_n)=\infty$ and $d\tt_n\left(U_n\cap (e_n,b_n)\right)=m((e_n,b_n))<\infty$. Particularly, for any $\epsilon>0$,
\[
	W_n\cap (b_n-\epsilon, b_n)\neq \emptyset.
\]
In other words, it will not happen that $b_n=b_m^n$ for some integer $m$ in \eqref{EQ3UNM}. If $b_n=\infty$, then $W_n$ may be trivial as in Example~\ref{EXA218}, i.e. $W_2=\{1\}$. Also possibly as in \cite[Remark~3.2]{LY14}, $W_n$ is not trivial in the sense that for any $L>a_n$,
\[
	W_n\cap (L, \infty)\neq \emptyset.
\]
Finally we note that
\[
K=\left(\bigcup_{n\geq 1}{W_n}\right) \bigcup \left( \bigcup_{n\geq 1} I_n\right)^c,
\]
where $\left(\bigcup_{n\geq 1}I_n\right)^c$ is an $m$-polar set relative to $X$ by Theorem~\ref{THM23}. Since $m(W_n)=0$ and $m\left(\left(\cup_{n\geq 1}I_n\right)^c\right)=0$, we obtain $m(K)=0$.
\end{remark}
}

The darning method introduced in \cite[\S3.2]{LY14} may also be applied to investigate the behavior of $(\EE,\mathcal{G})$, where $\mathcal{G}$ is the orthogonal complement \eqref{EQ3GFF} of $H^1_\mathrm{e}(\mathbb{R})$ in $\FF_\mathrm{e}$. Let
\[
\begin{aligned}
	&\mathcal{G}^n:=\mathcal{G}|_{I_n}=\left\{f\in \FF^n_\mathrm{e}: \frac{df}{d\tt_n}=0, d\tt_n\text{-a.e. on }U_n \right\}, \\
	&\EE^n(f,f)=\EE^{I_n}(f,f)=\frac{1}{2}\int_{I_n} \left(\frac{df}{d\tt_n}\right)^2d\tt_n,\quad f\in \mathcal{G}^n,
\end{aligned}
\]
where $\FF^n_\mathrm{e}$ is given by \eqref{EQ3FNE}. Further denote
\[
	\mathcal{G}^n_0:=\mathcal{G}^n\cap L^2(I_n)=\left\{f\in \FF^n: \frac{df}{d\tt_n}=0, d\tt_n\text{-a.e. on }U_n \right\}.
\]
Note that $U_n$ is open and expressed as \eqref{EQ3UNM}. Thus the function $f\in \mathcal{G}^n$ is a constant on $[a_m^n, b_m^n]$ for any integer $m\geq 1$. We need to exclude the case $d\tt_n(W_n)=0$, which gives us a trivial darning process. Thus we would make the following assumption in this section:
\begin{description}
\item[(H3)] $d\tt_n(W_n)>0$.
\end{description}
Define $r_n^-:=\inf\{x: x\in W_n\}, r_n^+:=\sup\{x: x\in W_n\}$ and
\[
	J_n:=\langle r^-_n, r^+_n\rangle,
\]
where $r^-_n\in J_n$ (resp. $r^+_n\in J_n$) if and only if $a_n\in I_n$ (resp. $b_n\in I_n$). Note that if $b_n$ (resp. $a_n$) is finite, then $r^+_n=b_n$ (resp. $r^-_n=a_n$) by Remark~\ref{RM34}. If $b_n=\infty$ (resp. $a_n=-\infty$), then $r^+_n$ (resp. $r^-_n$) may be finite and meanwhile $f=0$ on $[r^+_n, \infty)$ (resp. $(-\infty, r^-_n]$) for any $f\in \mathcal{G}^n_0$. Thus $(\EE^n, \mathcal{G}^n_0)$ on $L^2(I_n)$ can be identified with the one on $L^2(J_n)$. Then we have the following lemma. The proof is similar to \cite[Lemma~3.2]{LY14} and we omit it.

\begin{lemma}
The quadratic form $(\EE^n, \mathcal{G}^n_0)$ is a Dirichlet form on $L^2(J_n)$ in the wide sense, i.e. it satisfies all conditions of Dirichlet form except for the denseness of $\mathcal{G}^n_0$ in $L^2(J_n)$.
\end{lemma}

As stated in \cite[\S3.2]{LY14}, $(J_n, m|_{J_n}, \mathcal{G}^n_0, \EE^n)$ is a D-space that named by Fukushima in \cite{F71}. We introduced the darning method to find the regular representations of the D-spaces we explored in \cite{LY14}. In what follows, we shall describe the road map to attain the regular representation of $(J_n, m|_{J_n}, \mathcal{G}^n_0, \EE^n)$ via the darning method, but omit most details of the proof, since it is indeed similar to \cite{LY14}.

Recall that $e_n$ is a fixed point in $(a_n,b_n)$. We introduce the following transform on $J_n$ that collapses each open component $(a^n_m,b^n_m)$ with its endpoints of $U_n$ into a new point:
\[
	j_n(x):=\int_{e_n}^x 1_{W_n}(y)d\tt_n(y),\quad x\in J_n.
\]
If $a_n\in I_n$ (resp. $b_n\in I_n$), then $r^-_n=a_n$ and $r^{-*}_n:=j_n(r^-_n)>-\infty$ (resp. $r^+_n=b_n$ and $r^{+*}_n:=j_n(r^+_n)<\infty$). If $a_n\notin I_n$ and $a_n>-\infty$ (resp. $b_n\notin I_n$ and $b_n<\infty$), then $r^-_n=a_n$ and $r^{-*}_n=-\infty$ (resp. $r^+_n=b_n$ and $r^{+*}_n=\infty$). If $a_n=-\infty$ (resp. $b_n=\infty$), then $r^{-*}_n$ (resp. $r^{+*}_n$) may be finite or infinite. Denote
\[
	J^*_n:= \langle r^{-*}_n, r^{+*}_n\rangle,
\]
where $r^{-*}_n\in J^*_n$ (resp. $r^{+*}_n\in J^*_n$) if and only if $a_n\in I_n$ (resp. $b_n\in I_n$). Clearly, $j_n(J_n)=J^*_n$, $j_n$ is non-decreasing, and $j_n(x)=j_n(y)$ if and only if $x,y\in [a_m^n,b_m^n]$ for some integer $m$. The assumption \textbf{(H3)} guarantees that $J^*_n$ is a nontrivial interval.

We further introduce the image measure of $m|_{J_n}$ relative to $j_n$ on $J^*_n$:
\[
	m^*_n:= m|_{J_n}\circ j_n^{-1}.
\]
Note that $m^*_n$ is a Radon measure on $J^*_n$. Moreover, when $r^{-*}_n\in J^*_n$ (resp. $r^{+*}_n\in J^*_n$), it probably holds $m^*_n(\{r^{-*}_n\})>0$ (resp. $m^*_n(\{r^{+*}_n\})>0$). This situation only happens when $a_n$ (resp. $b_n$) is the left (resp. right) endpoint of $(a^n_m, b^n_m)$ for some integer $m$.

Since $f\in \mathcal{G}^n_0$ is a constant on $[a^n_m, b^n_m]$ for any $m\geq 1$, this function determines a unique function $\hat{f}$ on $J^*_n$ through a darning method:
\[
	\hat{f}\circ j_n=f.
\]
For any function $f\in \mathcal{G}^n_0\subset \FF^n$, it may be written as $f=g\circ \tt_n$ for some absolutely continuous function $g$ with $\int_{\tt_n(I_n)} g'(x)^2dx<\infty$. Particularly, $g$ is a constant on $[\tt_n(a^n_m), \tt_n(b^n_m)]$. Then $g$ determines a unique function $\hat{g}$ on $J^*_n$ via $\hat{g}\circ j'_n=g$, where
\[
	j'_n: \tt_n(J_n)\rightarrow J^*_n,\quad  x\mapsto \int_{\tt_n(e_n)}^{x} 1_{\tt_n(W_n)}(y)dy.
\]
Clearly, $\hat{f}=\hat{g}$. Hence
\[ \EE^n(f,f)=\frac{1}{2}\int_{I_n}\left(\frac{df}{d\tt_n}\right)^2d\tt_n=\frac{1}{2}\int_{\tt_n(I_n)}g'(x)^2dx=\frac{1}{2}\int_{J^*_n}\hat{g}'(x)^2dx=\frac{1}{2}\int_{J^*_n}\hat{f}'(x)^2dx.
\]
On the other hand, when $r^{-*}_n\notin J^*_n$ but $r^{-*}_n>-\infty$ (resp. $r^{+*}_n\notin J^*_n$ but $r^{+*}_n<-\infty$), $\EE^n(f,f)<\infty$ implies $\hat{f}(r^{-*}_n):=\lim_{x\downarrow r^{-*}_n}\hat{f}(x)$ (resp. $\hat{f}(r^{+*}_n):=\lim_{x\uparrow r^{+*}_n}\hat{f}(x)$) exists. We assert that $\hat{f}(r^{-*}_n)=0$ (resp. $\hat{f}(r^{+*}_n)=0$). We only treat the left endpoint $r^{-*}_n$. Note that $r^{-*}_n\notin J^*_n$ and $r^{-*}_n>-\infty$ indicate $a_n=-\infty$. If $r^-_n>-\infty$, we pointed out $f=0$ on $(-\infty, r^-_n]$ and thus $\hat{f}(r^{-*}_n)=0$. If $r^-_n=-\infty$, if follows that $f(-\infty):=\lim_{x\downarrow -\infty}f(x)$ exists, whereas $f\in L^2(I_n)$. Hence it holds $f(-\infty)=0$, which implies $\hat{f}(r^{-*}_n)=0$. Therefore, we are lead to define the quadratic form on $L^2(J^*_n, m^*_n)$:
\begin{equation}\label{EQ3GNF}
\begin{aligned}
	&\mathcal{G}^{n*}_0:=\{\hat{f}: f\in \mathcal{G}^n_0\}, \\
	&\EE^{n*}(\hat{f},\hat{g}):=\frac{1}{2}\int_{J^*_n}\hat{f}'(x)\hat{g}'(x)dx,\quad \hat{f},\hat{g}\in \mathcal{G}^{n*}_0.
\end{aligned}
\end{equation}
The following theorem is an analogue of \cite[Theorem~3.2]{LY14}.

\begin{theorem}\label{THM313}
The quadratic form $(\EE^{n*}, \mathcal{G}^{n*}_0)$ defined by \eqref{EQ3GNF} can be expressed as
\[
\begin{aligned}
	&\mathcal{G}^{n*}_0=H^1_{0,\mathrm{e}}(J^*_n)\cap L^2(J^*_n, m^*_n),\\
	&\EE^{n*}(\hat{f},\hat{g}):=\frac{1}{2}\int_{J^*_n}\hat{f}'(x)\hat{g}'(x)dx,\quad \hat{f},\hat{g}\in \mathcal{G}^{n*}_0,
\end{aligned}
\]
where
\[
H^1_{0,\mathrm{e}}(J^*_n)=\left\{\hat{f}\in H^1_\mathrm{e}(J^*_n): \hat{f}(r^{\pm *}_n)=0\text{ whenever }r^{\pm *}_n\notin J^*_n\text{ and }|r^{\pm *}_n|<\infty\right\}.
\]
Furthermore, a regular representation of D-space $(J_n, m|_{J_n}, \mathcal{G}^n_0, \EE^n)$ can be realized by the regular local Dirichlet form $(\EE^{n*}, \mathcal{G}^{n*}_0)$ on $L^2(J^*_n, m^*_n)$. Its associated diffusion process $X^{n*}$ is a Brownian motion $B^*$ on $J^*_n$ being time changed by its positive continuous additive functional with the Revuz measure $m^*_n$, where $B^*$ reflects at the finite endpoints $r^{\pm *}_n\in J^*_n$ and absorbs at the finite endpoints $r^{\pm *}_n\notin J^*_n$.
\end{theorem}

At the finite endpoints $r^{\pm *}_n\in J^*_n$, $X^{n*}$ is said to be slowly reflecting if $m^*_n(\{r^{\pm *}_n\})>0$ and instantaneously reflecting if $m^*_n(\{r^{\pm *}_n\})=0$ by \cite[Chapter~VII~(3.11)]{RY99}. The former case occurs if and only if $a_n$ (resp. $b_n$) is finite and $a_n$ (resp. $b_n$) is the left (resp. right) endpoint of $(a^n_m,b^n_m)$ for some integer $m$. At this time, $X^{n*}$ is also called a diffusion with sojourn in \cite{F14}.

{We end this section with two examples of darning processes.

\begin{example}
We first consider the regular Dirichlet extension of one-dimensional Brownian motion in Example~\ref{EXA215}. Note that it is irreducible and thus only has one invariant interval $I_1=\mathbb{R}$. Hereafter, we get rid of the subscript `$1$' for convenience and write $I=\mathbb{R}$. Moreover, $U=K^c, W=K$, where $K$ is the standard Cantor set in $[0,1]$.  Clearly, $U$ and $W$ satisfy \textbf{(H1)}, \textbf{(H2)} and \textbf{(H3)}. Recall that $\tt(x)=x+c(x)$, where $c$ is the standard Cantor function and
\[
\begin{aligned}
&\mathcal{G}_0=\{f\in \FF: \frac{df}{d\tt}=0, d\tt\text{-a.e. on }U\},\\
&\EE(f,f)=\frac{1}{2}\int_{\mathbb{R}}\left(\frac{df}{d\tt}\right)^2d\tt,\quad f\in \mathcal{G}_0,
\end{aligned}
\]
where $\FF=\{f\in L^2(\mathbb{R}): f\ll \tt, \EE(f,f)<\infty\}$. Clearly, for any $f\in \mathcal{G}_0$, $f=0$ on $(-\infty, 0]$ and $[1,\infty)$.

Since $I$ is open, we have
\[
	J=(0, 1).
\]
Take the fixed point $e=0$, and the darning transform $j$ is
\[
	j(x)=\int_0^x1_K(y)d\tt(y)=\int_0^x1_K(y)dc(y)=c(x),\quad x\in (0,1).
\]
Then $J^*=j(J)=(0,1)$ and $m^*=m|_{(0,1)}\circ j^{-1}$ is a fully supported Radon measure on $J^*$ with $m^*(J^*)=m(J)=1$. Note that the single point set of $J^*$ may be of positive $m^*$-measure. For example, $m^*(\{1/2\})=m([1/3, 2/3])=1/3$.  Furthermore,
\[
\begin{aligned}
&\mathcal{G}^*_0=H^1_{0,\mathrm{e}}((0,1))\cap L^2((0,1), m^*),\\
&\EE^*(f,g)=\frac{1}{2}\int_0^1 f'(x)g'(x)dx,\quad f,g\in \mathcal{G}^*_0,
\end{aligned}
\]
where $H^1_{0,\mathrm{e}}((0,1))=\{f\in H^1_\mathrm{e}((0,1)): f(0)=f(1)=0  \}$.
The associated darning process is a time-changed absorbing Brownian motion by $m^*$ on $(0,1)$.
\end{example}

\begin{example}
In this example, we show a darning process with sojourn at the boundary. Let $(\EE,\FF)$ be a regular Dirichlet extension of one-dimensional Brownian motion:
\[
	I_1=(-\infty, -1),\quad I_2=[-1, \infty)
\]
and $\tt_2(x)=x+c(x)$, where $c(x)$ is still the standard Cantor function with $c(x):=0$ for $x\leq 0$ and $c(x):=1$ for $x\geq 1$.

We only consider the restriction to $I_2$ of the orthogonal complement $\mathcal{G}$. Let $K$ be the standard Cantor set in $[0,1]$. Then
\[
	U_2=(-1,0) \cup \left([0,1]\setminus K \right) \cup (1,\infty),\quad
	W_2=\{-1\}\cup K.
\]
Clearly, $U_2$ and $W_2$ satisfy \textbf{(H1)}, \textbf{(H2)} and \textbf{(H3)}.
Since $-1\in I_2$, we have
\[
	J_2=[-1, 1).
\]
Take the fixed point $e_2=0$ and the darning transform is
\[
	j_2(x)=\int_0^x1_{K}(y)d\tt_2(y)=\int_0^x1_K(y)dc(y),\quad x\in [-1,1).
\]
Thus $J^*_2=j_2(J_2)=[0, 1)$ and $m_2^*=m|_{[-1,1)}\circ j_2^{-1}$ is a fully supported Radon measure on $[0,1)$. Particularly, $m_2^*(\{0\})=m([-1,0])=1$. Furthermore,
\[
\begin{aligned}
&\mathcal{G}^{2*}_0=H^1_{0,\mathrm{e}}([0,1))\cap L^2([0,1), m_2^*),\\
&\EE^{2*}(f,g)=\frac{1}{2}\int_0^1 f'(x)g'(x)dx,\quad f,g\in \mathcal{G}^{2*}_0,
\end{aligned}
\]
where $H^1_{0,\mathrm{e}}([0,1))=\{f\in H^1_\mathrm{e}([0,1)): f(1)=0  \}$. The associated darning process $X^{2*}$ is a Brownian motion $B^*$ on $[0,1)$ being time-changed by $m^*$, where $B^*$ reflects at $0$ and absorbs at $1$. Since $m_2^*(\{0\})=1>0$, $X^{2*}$ is a diffusion process with sojourn and slowly reflecting at $0$.
\end{example}
}

\section{{Structures of regular Dirichlet extensions: trace Dirichlet forms}}\label{SEC5}

In previous section we discuss the orthogonal complement $\mathcal{G}$ of one-dimensional Brownian motion in
extension space $(\EE,\FF_e)$. In this section we shall {only impose \textbf{(H1)} and \textbf{(H2)} of \S\ref{SEC42} and} discuss the orthogonal complement of the part Dirichlet form of $(\EE,\FF)$
on the open set $U$, or intuitively the biggest Brownian motion contained in $(\EE,\FF)$. The later complement is
called the trace of $(\EE,\FF)$ on $U^c$, which may be orthogonally decomposed into the former complement $\mathcal{G}$ and
the trace of Brownian motion on $U^c$.

The following lemma is similar to \cite[Lemma~2.2]{LY14},
which indicates that $X$ is a Brownian motion before leaving the open set $U$.

\begin{lemma}\label{LM34}
Let $(\frac{1}{2}\mathbf{D}_U, H^1_0(U))$ and $(\EE_U,\FF_U)$ be the part Dirichlet forms of $(\frac{1}{2}\mathbf{D}, H^1(\mathbb{R}))$  and $(\EE,\FF)$ on $U$. Then it holds that $(\EE_U,\FF_U)=(\frac{1}{2}\mathbf{D}_U, H^1_0(U))$.
\end{lemma}

\begin{proof}
Note that $H^1_0(U)\subset \FF_U$ and $\EE_U(f,f)=\frac{1}{2}\mathbf{D}_U(f,f)$ for any $f\in H^1_0(U)$ by Proposition~\ref{PRO25}. Thus it suffices to prove $\FF_U\subset H^1_0(U)$. Note that $\tt_n$ is a natural scale (i.e. $\tt_n(x)=x+c$ for some constant $c$) on $(a^n_m,b^n_m)$. This implies any function $f\in \FF_U$ is absolutely continuous on $(a^n_m,b^n_m)$ and $f(a^n_m)=f(b^n_m)=0$. It follows from \eqref{EQ2FFL} that
\[
\sum_{n,m\geq 1} \int_{(a^n_m,b^n_m)} f'(x)^2dx<\infty.
\]
Then we can conclude that $f$ is absolutely continuous on $\mathbb{R}$ and hence $f\in H^1_0(U)$.
\end{proof}

We now turn to the trace Dirichlet forms of $(\EE,\FF)$ and $(\frac{1}{2}\mathbf{D},H^1(\mathbb{R}))$ on $K$. To do that, we have to find a smooth measure supported on $K$.
For each $n$, $d\tt_n$ is a Radon measure on $I_n$ but not necessarily finite. Nevertheless, we can always take a finite measure $d\tt'_n$ equivalent to $d\tt_n$ if $I_n$ is finite. For example,
\[
d\tt'_n= \sum_{k\geq 1} \frac{C_n}{2^k\cdot d\tt_n\left([a_n+1/k, b_n-1/k]\right)}\cdot d\tt_n|_{[a_n+1/k, b_n-1/k]},
\]
where $C_n$ is some positive constant and we make the convention $0/0=0$. Particularly, we may choose $C_n$ so that $d\tt_n'(I_n)=b_n-a_n$. If $I_n$ is infinite, i.e. $I_n=\langle a_n,\infty), (-\infty, b_n\rangle $ or $\mathbb{R}$, we write $d\tt'_n:=d\tt_n$. Define  a measure
\begin{equation}\label{EQ3MNT}
\mu:=\sum_{n\geq 1}d\tt'_n|_{W_n} +\sum_{n\geq 1} (b_n-a_n)\cdot \left(\delta_{a_n}\cdot 1_{\{a_n\in I_n\}}+\delta_{b_n}\cdot 1_{\{b_n\in I_n\}}\right),
\end{equation}
where $W_n=I_n\setminus U_n$ and $\delta_{a_n}$ is the Dirac measure at $a_n$.
It can be seen from the following lemma that $\mu$ might be a suitable choice.

\begin{lemma}\label{LM46}
The measure $\mu$ given by \eqref{EQ3MNT} is a Radon smooth measure with the topological support $K$ relative to $(\frac{1}{2}\mathbf{D},H^1(\mathbb{R}))$ and $(\EE,\FF)$ respectively. Hence the quasi support of $\mu$ relative to $(\frac{1}{2}\mathbf{D},H^1(\mathbb{R}))$ is $K$. Furthermore, the quasi support of $\mu$ relative to $(\EE,\FF)$ can be taken as a finely closed q.e. version $K$.
\end{lemma}
\begin{proof}
Clearly, $\mu$ is a Radon measure on $\mathbb{R}$. Since the $m$-polar set relative to $(\frac{1}{2}\mathbf{D},H^1(\mathbb{R}))$ must be the empty set, it follows that $\mu$ is smooth relative to $(\frac{1}{2}\mathbf{D},H^1(\mathbb{R}))$. It is also smooth relative to $(\EE,\FF)$ since the $m$-polar sets relative to $(\EE,\FF)$ must be the subsets of $\left(\cup_{n\geq 1}I_n\right)^c$ and clearly $\mu\left(\left(\cup_{n\geq 1}I_n\right)^c\right)=0$.

Next, we prove the topological support of $\mu$ is $K$. Note that $K$ is closed and $\mu(K^c)=\mu(U)=0$. If $K'$ is another closed set and $\mu(K'^c)=0$, we assert that $K\subset K'$. Suppose that $x\in K\setminus K'$. Then $(x-\epsilon,x+\epsilon)\cap K'=\emptyset$ for some constant $\epsilon>0$. If $x\in W_n\cap (a_n,b_n)$ for some $n$, then it follows from \textbf{(H2)} that $\mu((x-\epsilon,x+\epsilon))\geq \mu((x-\epsilon,x+\epsilon)\cap W_n)>0$, which contradicts the fact $\mu(K'^c)=0$. Otherwise $(x-\epsilon, x]$ must contain a part with an endpoint of some interval $I_n$. When this endpoint belongs to $I_n$, clearly $\mu((x-\epsilon,x])>0$. When this endpoint does not belong to $I_n$, we have $d\tt_n(I_n\cap (x-\epsilon, x])=\infty$ whereas $d\tt_n(U_n\cap (x-\epsilon, x])\leq \epsilon$. This implies $d\tt'_n|_{W_n}((x-\epsilon,x])>0$ and thus $\mu((x-\epsilon, x+\epsilon))>0$, which also contradicts the fact $\mu(K'^c)=0$.

Since the fine topology relative to the one-dimensional Brownian motion is the same as the usual topology, we conclude that the quasi-support of $\mu$ relative to $(\frac{1}{2}\mathbf{D},H^1(\mathbb{R}))$ is also $K$. For the last assertion, we need only to prove \cite[Theorem~3.3.5~(b)]{CF12} for $F=K$. If $u\in \FF$ and $u=0$ q.e. on $K$, then $u(x)=0$ for any $x\in \cup_{n\geq 1}W_n$. This implies $u=0$ $\mu$-a.e. On the contrary, let $u\in \FF$ and $u=0$ $\mu$-a.e. We assert that $u(x)=0$ for any $x\in I_n\cap K=W_n$, which implies $u=0$ q.e. on $K$. In fact, assume $u(x)\neq 0$ for some $x\in W_n$. Since $u=0$ $\mu$-a.e., $x$ is not the endpoint of $I_n$. Note that $u|_{I_n}$ is continuous. Thus $u(y)\neq 0$ for any $y\in (x-\epsilon, x+\epsilon)$ with some constant $\epsilon>0$. However, $\mu((x-\epsilon,x+\epsilon))>0$ by \textbf{(H2)}, which contradicts the fact $u=0$ $\mu$-a.e. That completes the proof.
\end{proof}

\begin{remark}\label{RM36}
From Lemma~\ref{LM46} and \cite[Lemma~5.2.9~(iii)]{CF12}, we know that any Radon smooth measure $\mu'$, for example $\mu$ given by \eqref{EQ3MNT}, with the quasi support $K$ relative to $(\EE,\FF)$ or $(\frac{1}{2}\mathbf{D}, H^1(\mathbb{R}))$ always has the topological support $K$. The trace Dirichlet form induced by $\mu'$ is a regular Dirichlet form on $L^2(K,\mu')$ as asserted by \cite[Corollary~5.2.10]{CF12}. The choice of $\mu'$ is not essential in the sense of \cite[Theorem~5.2.15]{CF12}.
\end{remark}

Trace Dirichlet form on some appropriate set $F$ characterizes the `trace' of the associated Markov process left on $F$. Precisely, given a symmetric Markov process $Y$ with the regular Dirichlet form $(\EE^Y, \FF^Y)$ on the state space $E$ and $F\subset E$ a closed subset with the positive capacity, let $\nu$ be a Radon smooth measure on $E$ with the same topological and quasi support $F$ and $\sigma_F^Y$ be the hitting time of $F$ relative to $Y$. Set for any $f\in \FF^Y_\mathrm{e}$,
\[
	\mathbf{H}^Y_Ff(x):=\mathbf{E}_x f(Y_{\sigma^Y_F}),\quad x\in E.
\]
Then
\[
\begin{aligned}
&\check{\FF}^Y:=\left\{ \varphi\in L^2(F,\nu): \varphi=f\; \nu\text{-a.e. on }F \text{ for some } f\in \FF^Y_\mathrm{e}\right\},   \\
&\check{\EE}^Y(\varphi, \varphi):=\EE^Y(\mathbf{H}^Y_Ff, \mathbf{H}^Y_Ff), \quad \varphi\in \check{\FF}^Y, \varphi=f\; \nu\text{-a.e. on }F, f\in \FF_\mathrm{e}
\end{aligned}
\]
is called the trace Dirichlet form of $(\EE^Y,\FF^Y)$ induced by $\nu$. It is a regular Dirichlet form on $L^2(F,\nu)$ as in Remark~\ref{RM36}.
We refer the details of trace Dirichlet forms and their Feller measures to \cite{CFY06}, \cite[\S5.5]{CF12} and \cite{LY14}.

Denote the trace Dirichlet forms of $(\frac{1}{2}\mathbf{D}, H^1(\mathbb{R}))$ and $(\EE,\FF)$ induced by the measure $\mu$, given by \eqref{EQ3MNT}, by $(\frac{1}{2}\check{\mathbf{D}}, \check{H}^1)$ and $(\check{\EE},\check{\FF})$, respectively. They are both regular and recurrent (Cf. \cite[Theorem~5.2.5]{CF12}) Dirichlet forms on $L^2(K,\mu)$. The associated Hunt processes are denoted by $\check{B}=(\check{B}_t)_{t\geq 0}$ and $\check{X}=(\check{X}_t)_{t\geq 0}$. Their extended Dirichlet spaces are naturally denoted by $\check{H}^1_\mathrm{e}$ and $\check{\FF}_\mathrm{e}$. We have (Cf. \cite[Theorem~5.2.15]{CF12})
\[
\begin{aligned}
&\check{H}^1_\mathrm{e}=H^1_\mathrm{e}(\mathbb{R})|_K = \left\{f|_K: f\in H^1_\mathrm{e}(\mathbb{R}) \right\}, \\
&\check{\FF}_\mathrm{e}=\FF_\mathrm{e}|_K =\left\{f|_{K}: f\in \FF_\mathrm{e}\right\}.
\end{aligned}
\]
Recall that $U_n, W_n$ are defined by \eqref{EQ3UNX} and $U_n$ is expressed as \eqref{EQ3UNM}. We now state the main result in this section, which is similar to \cite[Theorem~2.1]{LY14} in the sense that they both give an example that a pure jump Dirichlet form is a proper regular Dirichlet subspace of a Dirichlet form with strongly local part.

\begin{theorem}\label{THM38}
Let $(\frac{1}{2}\check{\mathbf{D}}, \check{H}^1)$ and $(\check{\EE},\check{\FF})$ be given above. Then $(\frac{1}{2}\check{\mathbf{D}}, \check{H}^1)$ is a proper regular Dirichlet subspace of $(\check{\EE},\check{\FF})$, i.e.
\[
\check{H}^1\subset \check{\FF},\quad \check{\EE}(\varphi, \varphi)=\frac{1}{2}\check{\mathbf{D}}(\varphi, \varphi),\quad \varphi\in \check{H}^1.
\]
Furthermore for any $\varphi\in \check{\FF}_\mathrm{e}=\FF_\mathrm{e}|_{K}$,
\begin{equation}\label{EQ3EVV}
	\check{\EE}(\varphi, \varphi)=\frac{1}{2}\sum_{n\geq 1}\int_{W_n}\left(\frac{d\varphi}{d\tt_n}\right)^2d\tt_n+\frac{1}{2}\sum_{n\geq 1}\sum_{m\geq 1} \frac{(\varphi(a_m^n)-\varphi(b^n_m))^2}{|a_m^n-b_m^n|},
\end{equation}
and for any $\varphi\in \check{H}^1_\mathrm{e}=H^1_\mathrm{e}(\mathbb{R})|_{K}$,
\begin{equation}\label{EQ3DVV}
\frac{1}{2}\check{\mathbf{D}}(\varphi, \varphi)= \frac{1}{2}\sum_{n\geq 1}\sum_{m\geq 1} \frac{(\varphi(a_m^n)-\varphi(b^n_m))^2}{|a_m^n-b_m^n|}.
\end{equation}
\end{theorem}
\begin{proof}
The first assertion is similar to \cite[Theorem~2.1~(1)]{LY14} by Lemma~\ref{LM34}. The trace formula \eqref{EQ3DVV} of one-dimensional Brownian motion on $K$ can be formulated as in the proof of \cite[Theorem~2.1~(2)]{LY14} and we further remark that $m(K)=0$.

Now we prove the trace formula \eqref{EQ3EVV}. Note that the trace Dirichlet form $(\check{\EE},\check{\FF})$ corresponds to a time-changed Markov process $\check{X}$ of $X$. Precisely, let $(A_t)_{t\geq 0}$ be the associated positive continuous additive functional of $\mu$ relative to $X$ and $\tau_t$ be its right continuous inverse, i.e. $\tau_t:=\inf\{s>0: A_s>t\}$ for any $t\geq 0$. Then
\[
	\check{X}_t=X_{\tau_t}, \quad t\geq 0.
\]
On the other hand, a subset $F\subset K$ is $\check{\EE}$-polar if and only if $F$ is $\EE$-polar as a subset of $\mathbb{R}$ (Cf. \cite[Theorem~5.2.8]{CF12}). This implies that $(\bigcup_{n\geq 1}I_n)^c$ is an $\check{\EE}$-polar set. Since for each $n$, $I_n$ is an invariant set of $X$, it follows that $I_n\cap K=W_n$ is an invariant set of $\check{X}$ in the sense that
\[
	\mathbf{P}^{\check{X}}_x(\check{X}_t\in W_n, \forall t)=1, \quad x\in W_n,
\]
where $\mathbf{P}^{\check{X}}_x$ is the probability measure of $\check{X}$ starting from $x$. Particularly, $(\check{\EE}^{W_n}, \check{\FF}^{W_n})$ is the trace Dirichlet form of $(\EE^{I_n},\FF^{I_n})$ induced by $\mu|_{W_n}$. It suffices to prove that for any $\varphi\in \FF^{I_n}_\mathrm{e}|_{W_n}$,
\begin{equation}\label{EQ3EWN}
	 \check{\EE}^{W_n}(\varphi,\varphi)=\frac{1}{2}\int_{W_n}\left(\frac{d\varphi}{d\tt_n}\right)^2d\tt_n+\frac{1}{2}\sum_{m\geq 1}\frac{(\varphi(a_m^n)-\varphi(b^n_m))^2}{|a_m^n-b_m^n|},
\end{equation}
since the trace formula \eqref{EQ3EVV} may then be attained from Corollary~\ref{COR210}. Indeed, note that for any $f\in \FF^{I_n}_\mathrm{e}$, the energy measure (Cf. \cite[(4.3.8)]{CF12}) of $f$ is equal to
\[
	\mu_{\langle f\rangle}=\left(\frac{df}{d\tt_n}\right)^2d\tt_n,
\]
which can be formulated by an approach similar to \cite[(2.2)]{LY14}. The Feller measure corresponding to $(\check{\EE}^{W_n}, \check{\FF}^{W_n})$ is deduced by the same idea as in the proof of \cite[Theorem~2.1]{LY14} since $X^{I_n}$ is a Brownian motion before leaving $U_n$ by Lemma~\ref{LM34}. Therefore we obtain \eqref{EQ3EWN} which is similar to \cite[Theorem~2.1]{LY14}. That completes the proof.
\end{proof}

{Although the main ideas to prove the above theorem come from \cite[Theorem~2.1]{LY14}, we still need to point out the different significance of Theorem~\ref{THM38}. In \cite[Theorem~2.1]{LY14}, the state space $F$ of the trace Dirichlet forms must be of positive Lebesgue measure to guarantee that the associated regular Dirichlet subspace $(\EE^{(s)},\FF^{(s)})$ of the one-dimensional Brownian motion is a proper subspace. This fact causes that the strong local part of one of the trace Dirichlet forms in \cite[Theorem~2.1]{LY14} never disappears. When coming back to the above theorem, we find that the strongly local part may disappear in some special situation (i.e. $W_n$ is of zero $d\tt_n$-measure) and then an interesting phenomena shows up. }

\begin{corollary}\label{COR39}
Let $(\frac{1}{2}\check{\mathbf{D}}, \check{H}^1)$ and $(\check{\EE},\check{\FF})$ be given in Theorem~\ref{THM38}. Assume that $d\tt_n(W_n)=0$, in other words, $I_n$ is closed and $\tt_n$ is the natural scale function on $I_n$, for each $n\geq 1$. Then $(\frac{1}{2}\check{\mathbf{D}}, \check{H}^1)$ is a proper regular Dirichlet subspace of $(\check{\EE},\check{\FF})$ on $L^2(K,\mu)$. Furthermore, for any $\varphi\in \check{\FF}_\mathrm{e}$,
\begin{equation}\label{EQ3EVVS}
\check{\EE}(\varphi, \varphi)=\frac{1}{2}\sum_{n\geq 1}\frac{(\varphi(a_n)-\varphi(b_n))^2}{|a_n-b_n|},
\end{equation}	
and for any $\varphi\in \check{H}^1_\mathrm{e}$,
\begin{equation}\label{EQ3DVVS}
\frac{1}{2}\check{\mathbf{D}}(\varphi, \varphi)=\frac{1}{2}\sum_{n\geq 1}\frac{(\varphi(a_n)-\varphi(b_n))^2}{|a_n-b_n|}.
\end{equation}
\end{corollary}

{The proof of Corollary~\ref{COR39} is trivial by Theorem~\ref{THM38}. Note that if $(\EE,\FF)$ is such a regular Dirichlet extension in this corollary, then its associated diffusion process is a reflected Brownian motion on each closed interval $I_n$. An example is given in Example~\ref{EXA218}, in which $K$ is the standard Cantor set in $[0,1]$.

The above corollary partially answers a problem in which we have been interested and studied for years. We know from Theorem~\ref{THM21} that if $(\EE^1,\FF^1)$
is a regular Dirichlet subspace of $(\EE^2,\FF^2)$, then the jumping and killing measures in their Beurling-Deny decompositions are the same.
Moreover, given a regular Dirichlet form, its killing part and `big jump' part do not play a role in producing a proper regular Dirichlet subspace
as described in \cite[\S2.2.3]{LY15} and \cite{LY15-2} respectively. For a strongly local Dirichlet form, many examples including
\cite{FFY05, FHY10, FL15, LY15-3} hint that it should always have proper regular Dirichlet subspaces. However we have not found any result to illustrate
 how the `small jump' part plays a role when concerning the regular Dirichlet subspaces. 
For the first time Corollary~\ref{COR39} gives us an example that a pure jump Dirichlet form has a proper regular Dirichlet subspace. This encourages us to keep going in this direction.

On the other hand, the jumps of a Hunt process are described by its L\'evy system denoted by $(N, H)$ in \cite{W64}, where $N(x,dy)$ is a kernel on the state space and $H$ is a positive continuous additive functional. We know that all  L\'evy systems of a symmetric Hunt process are equivalent in the sense that if $(N', H')$ is another L\'evy system, then $N(x,dy)\mu_H(dx)=N'(x,dy)\mu_{H'}(dx)$, where $\mu_{H}$ and $\mu_{H'}$ are the Revuz measures of $H$ and $H'$ respectively. Therefore, Corollary~\ref{COR39} also conduces to the following. }

\begin{corollary}
There exist two different symmetric pure jump Hunt processes that have the equivalent L\'evy systems.
\end{corollary}

The following corollary gives us an intuitive understanding of the differences between the two regular Dirichlet forms in Corollary~\ref{COR39}.

\begin{corollary}\label{COR311}
Let $(\frac{1}{2}\check{\mathbf{D}}, \check{H}^1)$ and $(\check{\EE},\check{\FF})$ be the regular Dirichlet forms on $L^2(K,\mu)$ in Corollary~\ref{COR39}. Then the following assertions hold.
\begin{itemize}
\item[(1)] $(\frac{1}{2}\check{\mathbf{D}}, \check{H}^1)$ is irreducible and for any $x, y\in K$,
\begin{equation}\label{EQ3PBX}
	\mathbf{P}^{\check{B}}_x(\sigma_y<\infty)>0,
\end{equation}
where $\mathbf{P}^{\check{B}}_x$ is the probability measure of $\check{B}$ starting from $x$ and $\sigma_y$ is the first hitting time of $\{y\}$ relative to $\check{B}$.
\item[(2)] $(\check{\EE},\check{\FF})$ is not irreducible. For each $n$ such that $a_n$ and $b_n$ are finite, $\{a_n,b_n\}$ is an invariant set of $(\check{\EE},\check{\FF})$ and the associated Hunt process $\check{X}$ only jumps between $a_n$ and $b_n$. Furthermore, $K\setminus \{a_n, b_n:  a_n>-\infty ,b_n<\infty, n\geq 1\}$ is $\check{\EE}$-polar.
\end{itemize}
\end{corollary}
\begin{proof}
The second assertion is obvious from the proof of Theorem~\ref{THM38}. We only prove (1). Note that $(\frac{1}{2}\check{\mathbf{D}}, \check{H}^1)$ is recurrent by \cite[Theorem~5.2.5]{CF12}. Let $\varphi=f|_K\in \check{H}^1_\mathrm{e}$ such that $\check{\mathbf{D}}(\varphi, \varphi)=0$. It follows that
\[
	\mathbf{D}(\mathbf{H}^B_Kf, \mathbf{H}^B_Kf)=0.
\]
Thus $\mathbf{H}^B_Kf\equiv C$ for some constant $C$ and $\varphi=f|_{K}=\mathbf{H}^B_Kf|_K\equiv C$. From \cite[Theorem~5.2.16]{CF12} we obtain that $(\frac{1}{2}\check{\mathbf{D}}, \check{H}^1)$ is irreducible. On the other hand, \cite[Theorem~5.2.8]{CF12} implies the $\frac{1}{2}\check{\mathbf{D}}$-polar set must be the empty set. Thus \eqref{EQ3PBX} follows from \cite[Theorem~3.5.6~(1)]{CF12}.
\end{proof}

{Corollary~\ref{COR311} shows us some interesting behavior of a Markov process associated with a Dirichlet form. Since the Feller measures in \eqref{EQ3EVVS} and \eqref{EQ3DVVS} are supported on
\[
\left\{(a_n,b_n), (b_n, a_n): a_n>-\infty, b_n<\infty, n\geq 1\right\},
\]
it seems that the Markov processes $\check{B}$ and $\check{X}$ only jump  between $a_n$ and $b_n$. Actually $\check{X}$ does jump this way. However, the trace $\check{B}$ of Brownian motion will hit any point in $K$ with positive probability. In other words, the motions of $\check{B}$ happen at where its potential energy is zero. These motions are not reflected in the energy form (i.e. Dirichlet form) but in its Dirichlet space. Recall that $\check{H}^1_\mathrm{e}$ is the restriction of $H^1_\mathrm{e}(\mathbb{R})$ to $K$ which are composed all by continuous functions, whereas $\check{\FF}_\mathrm{e}$ is the restriction of $\FF_\mathrm{e}$ to $K$ which is much bigger and contains many discontinuous functions. }



\end{document}